\documentclass[12pt,makeidx]{amsart}
\usepackage{amssymb}
\usepackage{color,titletoc}

\usepackage{url,titletoc,enumerate}

\makeindex

\usepackage[usenames,dvipsnames]{xcolor}
\usepackage[pagebackref,colorlinks,linkcolor=BrickRed,citecolor=OliveGreen,urlcolor=darkgray,hypertexnames=true]{hyperref}

    \def\la{\lambda}

\makeatletter
\newcommand{\mycontentsbox}{%
{
\addtolength{\parskip}{-2.3pt}
\tableofcontents}}
\def\enddoc@text{\ifx\@empty\@translators \else\@settranslators\fi
\ifx\@empty\addresses \else\@setaddresses\fi
\newpage\mycontentsbox 
}
\makeatother

  \newcommand{\usim}{\stackrel{u}{\thicksim}}

  \usepackage{mathrsfs,accents}

  \def\barvb{block-Arveson boundary}
  
  \def\BarvB{Block-Arveson Boundary}

\numberwithin{equation}{section}

\newtheorem{thm}{Theorem}[section]
\newtheorem{theorem}[thm]{Theorem}

\newtheorem{cor}[thm]{Corollary}
\newtheorem{lemma}[thm]{Lemma}

\newtheorem{prop}[thm]{Proposition}

\theoremstyle{definition}

\newtheorem{conj}         [thm]{Conjecture}

\newtheorem{remark}[thm]{Remark}

\newtheorem{example}[thm]{Example}

  \newcommand{\vertiii}[1]{{\left\vert\kern-0.25ex\left\vert\kern-0.25ex\left\vert #1
    \right\vert\kern-0.25ex\right\vert\kern-0.25ex\right\vert}}

\def\CR{ \color{red} }
\def\tX{\tilde X}
\def\pt{\partial}

\def\euc{\partial^{\rm Euc}}

\def\mco{\textrm{\rm co$^{\rm mat}$}}
\def\cmco{\overline{\textrm{\rm co}}^{\rm mat}}

\def\ben{\begin{enumerate} }%
\def\een{\end{enumerate}}

\def\blem{\begin{lemma}}
\def\elem{\end{lemma}}

\def\bem{\begin{pmatrix}}
\def\eem{\end{pmatrix}}

\def\beq{\begin{equation}}
\def\eeq{\end{equation}}

\renewcommand{\qedsymbol}{\rule[.12ex]{1.2ex}{1.2ex}}
\renewcommand{\subset}{\subseteq}
\renewcommand{\supset}{\supseteq}

\newcommand{\cEm}{\pt^{\rm mat}}

\newcommand{\cEv}{\pt^{\rm van}}

 \def\tX{\tilde X}

\newcommand{\df}[1]{{\bf{#1}}{\index{#1}}}

\def\arv{\partial^{\rm arv}}

\def\arvb{Arveson boundary}

\def\tY{{\tilde Y}}
\def\tS{{\tilde S}}
\def\tW{{\tilde W}}

\def\cC{\mathcal C}

\def\cK{\mathscr K}
\def\cD{\mathcal D}

\def\al{\alpha}
\def\be{\beta}
\def\ga{\gamma}
\def\de{\delta}
\def\epsilon{\varepsilon}

\def\bep{\proof}
\def\eep{\qed}
\def\bec{\begin{conj}}
\def\eec{\end{conj}}
\def\bex{\begin{example}}
\def\eex{\end{example}}
\def\de{\delta}

\def\mS{{\mathbb S}}
\def\smatmg{\mathbb S_m^g}

\def\smatg{\mathbb S^g}

\def\smatng{\mathbb S_n^g}

\def\smatmg{\mathbb S_m^g}
\def\smatdg{\mathbb S_d^g}
\def\matng{(\mathbb C^{n\times n})^g}
\def\smatg{\mathbb S^g}

\def\cS{\mathcal S}
\def\cH{\mathscr H}

\def\hTV{{ \widehat{\cD_p} }}

\DeclareMathOperator{\rg}{rg}

\def\R{ {\mathbb{R}} }
\def\C{ {\mathbb{C}} }

\def\N{ {\mathbb{N}} }

\def\La{\Lambda}

\def\pt{\partial}
\def\la{\lambda}
\def\tY{\tilde{Y}}
\def\pt{\partial}
\def\epsilon{\varepsilon}

\def\Om{\Omega}

\def\Ga{{\Gamma}}

\def\La{\Lambda}

\def\Na{\mathscr N}
\def\smatgpg{\mathbb S_{g+1}^g}
\def\cT{{\mathcal T}}
\def\smatgone{\mathbb S^{g+1}}
\def\bfo{\mathbf 1}
\def\bA{\mathbf A}
\def\bS{\mathbf S}
\def\free{\partial^{\operatorname{abs}}}
\def\comat{\operatorname{co}^{\operatorname{mat}}}
\def\cW{\mathcal W}

\title[Extreme points of matrix convex sets
 and dilation theory]{Extreme points of matrix convex sets,\\[1mm] free spectrahedra  and dilation theory}

\author[E. Evert]{Eric Evert}
\address{Eric Evert, Group Science, Engineering and Technology\\
 KU Leuven Kulak \\
  E. Sabbelaan 53, 8500 Kortrijk, Belgium \\
  and
  \newline
   Electrical Engineering ESAT/STADIUS\\
  KU Leuven \\
  Kasteelpark Arenberg 10, 3001 Leuven, Belgium
   }
   \email{eric.evert@kuleuven.be}

\author[J.W. Helton]{J. William Helton${}^1$}
\address{J. William Helton, Department of Mathematics\\
  University of California \\
  San Diego}
\email{helton@math.ucsd.edu}
\thanks{${}^1$Research supported by the National Science Foundation (NSF) grant
DMS 1201498, and the Ford Motor Co.}

\author[I. Klep]{Igor Klep${}^{2}$}
\address{Igor Klep, Department of Mathematics,
 University of Ljubljana, Slovenia}
\email{igor.klep@fmf.uni-lj.si}
\thanks{${}^2$Supported by the Marsden Fund Council of the Royal Society of New Zealand. Partially supported by
the Slovenian Research Agency grants P1-0222 and L1-6722.}

\author[S. McCullough]{Scott McCullough${}^3$}
\address{Scott McCullough, Department of Mathematics\\
  University of Florida\\ Gainesville 
   }
   \email{sam@math.ufl.edu}
\thanks{${}^3$Research supported by the NSF grant DMS-1361501.}

\date{\today}

\begin{document}

\begin{abstract}
For matrix convex sets a  unified  geometric interpretation of notions of extreme points
and of Arveson boundary points
is given.  These notions include, in increasing order of strength, the core notions of ``Euclidean'' extreme points, ``matrix'' extreme points, and ``absolute'' extreme points.
A seemingly different notion, the ``Arveson boundary'',
has by contrast a dilation theoretic flavor. An Arveson boundary point is
an analog of a (not necessarily irreducible) boundary representation for an operator system.  This article provides and explores dilation theoretic formulations for the above notions of extreme points.\looseness=-1

 The scalar solution set of a linear matrix inequality (LMI) is known as a spectrahedron.  The matricial solution set of an LMI is a free spectrahedron. Spectrahedra (resp. free spectrahedra) lie between general convex sets (resp. matrix convex sets) and convex polyhedra (resp. free polyhedra).  As applications
of our theorems on extreme points, it is shown the polar dual of a matrix convex set  $K$  is generated, as a matrix convex set,  by finitely many Arveson boundary points if and only if  $K$ is a free spectrahedron; and if the polar dual of a free spectrahedron $K$ is again a free spectrahedron, then at the scalar level $K$ is a polyhedron.

{\CR This version of the manuscript has been updated to address two dropped hypotheses and minor typos found in previous versions. The flaws were in Theorem \ref{thm:intromainext} \eqref{it:Euclidean-geometric} and Proposition \ref{prop:polyhedron} and were pointed out to us by Benjamin Passer and Tom-Lukas Kriel, respectively. We have added the correct hypotheses in the affected results with proofs, an example, and discussion in a new section, Section \ref{sec:RealFreeDuals}. 
} 
\end{abstract}

\maketitle

\setcounter{tocdepth}{3}
\contentsmargin{2.55em}
\dottedcontents{section}[3.8em]{}{2.3em}{.4pc}
\dottedcontents{subsection}[6.1em]{}{3.2em}{.4pc}
\dottedcontents{subsubsection}[8.4em]{}{4.1em}{.4pc}

\section{Introduction}
Spectrahedra, the solution sets of linear matrix inequalities (LMIs), play a central role in semidefinite programming, convex optimization and in real algebraic geometry \cite{BPR13,Nem06}.  They also figure prominently in the study of determinantal representations \cite{Bra11,NT12,Vin93}, the solution of the Lax conjecture \cite{HV07} and in the solution of the Kadison-Singer paving conjecture \cite{MSS15}.  The use of LMIs is a major advance in systems engineering in the past two decades \cite{BGFB94,SIG97}. Free spectrahedra, obtained by substituting matrix tuples instead of scalar tuples into an LMI, arise canonically in the theory of operator algebras, systems and spaces and the theory of matrix convex sets. Indeed, free spectrahedra are the prototypical examples of matrix convex sets over $\R^g$.  They also appear in  systems engineering, particularly in problems governed by a signal flow diagram (see \cite{dOHMP09}).

Extreme points are an important topic in convexity; they lie on the boundary of a convex set and capture
many of its properties \cite{Bar02}.
For a spectrahedron Ramana-Goldman
\cite{RG95}
 gave a basic characterization
of its Euclidean or classical extreme points.
For  matrix convex sets,
 operator algebras, systems, and operator spaces,
 it is natural to consider {\it quantized} analogs of the notion of an extreme point.
In \cite{WW99} the notion of a matrix extreme point of a compact matrix convex set $K$  was introduced and their result that the matrix extreme points span in the sense that their  closed matrix convex hull is $K$ (see also \cite{F04}) is now a foundation of the theory of matrix convex sets.
However, a proper subset of matrix extreme points might also have the spanning property.
One smaller class (introduced by Kleski \cite{Kls+}),  we call absolute extreme points,
 is closely related to a highly classical object, the Arveson boundary.
 \looseness=-1

 For operator algebras, systems and spaces {\it in infinite dimensions}
 Arveson's notion
\cite{Arv69} of an (irreducible) boundary representation
 (introduced as a noncommutative analog of peak points of function algebras)  is entirely satisfactory  \cite{Ham79,DM05,Arv08,DK15,FHL+} in that they span
the set  of which they are the boundary.
For matrix convex sets generally and free spectrahedra in particular,
 where the action takes place at  matrix levels and does not pass to operators, the situation is less clear. In the finite-dimensional context it is not known whether there are sufficiently many Arveson boundary points
 (or absolute extreme points) of a set to span the set.
 Indeed, the issue of whether there is a  natural notion of
quantized extreme points
for matrix convex sets that
is minimal (w.r.t.~spanning) remains unresolved (see for instance the discussion in \cite{F04}).
Fritz, Netzer and Thom \cite{FNT+}
use extreme points to investigate when
 an abstract operator system has a finite-dimensional concrete realization.\looseness=-1

In this article, in the context of matrix convex sets over $\R^g$, we provide geometric  unified interpretations of Arveson boundary points,
absolute extreme points, matrix extreme points and Euclidean  extreme points, giving them all dilation-theoretic interpretations (see Theorem \ref{thm:intromainext}). This theory of extreme points occupies the majority of the paper.

Next we give some applications of this theory.
We establish, in Theorem \ref{thm:weak-kleski} an analog of a result of Kleski \cite[Corollary 2.4]{Kls+}: a matrix convex set $K$ over $\R^g$ is spanned by finitely many of its Arveson boundary points  if and only if the polar dual $K^\circ$ is a free spectrahedron. As a consequence, in Corollary \ref{cor:hardtobepolar} we show if the polar dual of a  free spectrahedron $K$ is again a free spectrahedron, then at the scalar level $K$ is a polyhedron.  Further we show the spin disk \cite{HKMS+,DDSS+} in two variables provides a  non-trivial example of a free spectrahedron that is spanned by its Arveson boundary.
In another direction, we show using the Arveson boundary
that a natural construction of a matrix convex hull fails.

In the remainder of this introduction, we state, with some precision, our main results along the way introducing the necessary notations and definitions.

\subsection{Notation}
Given positive integers $g$ and $n$, let  \df{$\smatng$}
 denote the set of $g$-tuples $X=(X_1,\dots,X_g)$ of complex $n\times n$ self-adjoint matrices
  and let \df{$\smatg$} denote the sequence $(\smatng)_n$.
  A subset $\Gamma\subset \smatg$ is a sequence
  $\Gamma=(\Gamma(n))_n$ such that $\Gamma(n)\subset \smatng$ for each $n$.
  The set $\Gamma$  \df{closed with respect to direct sums} if for each pair of positive integers $m$ and $n$ and each
 $X\in \Gamma(n)$ and   $Y\in \Gamma(m)$,
\begin{equation}
 \label{eq:dirsums}
   X\oplus Y:= \big( \begin{pmatrix} X_1 & 0 \\ 0 & Y_1\end{pmatrix}, \dots, \begin{pmatrix} X_g & 0 \\ 0 & Y_g \end{pmatrix} \big) \in \Gamma(n+m).
\end{equation}
 Likewise $\Gamma$ is \df{closed with respect to unitary similarity}  if
\[
 U^* X U := (U^* X_1 U, \dots, U^* X_g U)\in\Gamma(n)
\]
 for each positive integer $n$, each $n\times n$ unitary matrix $U$ and each $X\in \Gamma(n)$.  A subset  $\Gamma\subset\smatg$ is a \df{graded set} if it is closed with respect to  direct sums and it is a \df{free set} if it is also closed with respect to unitary similarity.
 Often naturally occurring free sets are also closed with respect to restrictions to reducing subspaces. A free set $\Gamma$ is
\df{fully free} if it is
 \df{closed with respect to reducing subspaces}: if $X\in\Gamma(n)$ and $\cH\subset \C^n$ is reducing for $X$ of dimension $m$ and
the isometry
 $V:\C^m\to \C^n$ has range $\cH$, then  $V^*XV \in \Gamma(m).$  Free semialgebraic sets (see \cite{Lasse,HM12} for instance), namely the positivity sets of  free matrix-valued symmetric polynomials are fully free.
The set $\Gamma\subset\smatg$ is
 {\bf bounded}\index{uniformly bounded} \index{bounded, uniformly} if there is a $C\in\R_{>0}$ such that $C-\sum X_j^2 \succeq 0$
  for all $X\in\Gamma$.   We call $\Gamma$ \df{closed} (resp. \df{open}, \df{compact})   if each $\Gamma(n)\subseteq\smatng$ is closed (resp. open, compact).
 We refer the reader to
 \cite{Voi10,KVV+,MS11,Pope10,AM+,BB07,BKP16}
for a systematic study of free sets and free function theory.\looseness=-1

\subsection{Matrix convex sets and extreme points}
A  tuple $X\in\smatng$ is a \df{matrix convex combination} of
 tuples $Y^1,\dots,Y^N$ with $Y^\ell \in\mathbb S_{n_\ell}^g$ if
 there exist  $V_\ell:\mathbb C^n \to\mathbb C^{n_\ell}$ with
\beq
\label{eq:SumY}
  X = \sum_{j=1}^N V_\ell ^* Y^\ell  V_\ell
\qquad\text{and}\qquad
 \sum_{j=1}^N V_\ell^* V_\ell  = I_n.
\eeq
 The matrix convex combination \eqref{eq:SumY}
is \df{proper} provided each $V_\ell$ is surjective. In this case, $n\ge n_\ell$ for each $\ell$.
 We will say that a convex combination of the form \eqref{eq:SumY} is \df{weakly proper} if all of the $V_j$ are nonzero.

A graded set $K$ is \df{matrix convex} if it is closed
under matrix convex combinations.
Equivalently, $K$ is matrix convex if it is graded and for each pair of positive integers $m\le n$, each $X\in K(n)$ and each isometry $V:\C^m\to \C^n$, the tuple $V^*XV\in K(m)$; i.e., $K$ is closed with respect to \df{isometric conjugation}. In particular, a matrix convex set is a free set.

 Suppose $K$ is a free set. A tuple  $X \in K(n)$  is a \df{matrix extreme point}
  of the matrix convex set $K$
 if whenever it is represented as a proper matrix combination of the form
\eqref{eq:SumY} with $Y^\ell \in K(n_\ell)$,
then $n=n_\ell$ and  $X \usim  Y^\ell$ for each $\ell$.
(We use $A\usim B$ to denote $A$ and $B$ are unitarily equivalent.)
 A tuple  $X \in K(n)$  is an \df{absolute extreme point}  (a {\it boundary point} in the terminology of \cite{Kls+})  of $K$
 if whenever it is represented as a weakly proper matrix combination of the form \eqref{eq:SumY},  then
  for each $j$ either $n_j \leq n$ and $X \usim  Y^j$ (and hence $n_j=n$),
 or $n_j>n$ and there exists a $Z^j  \in K$  such that $Y^j \usim  X\oplus Z^j$.

There is a different, dilation-theoretic viewpoint of matrix extreme points.
 Given  a $g$-tuple $\al$ of $n\times m$ matrices and $\be\in\smatmg,$ let
\begin{equation}
 \label{eq:Z}
 Z =\begin{pmatrix} X & \al\\ \al^* & \be\end{pmatrix}.
\end{equation}
As a canonical  geometric notion of an Arveson boundary representation, we say $X$ is an \df{Arveson boundary point} of $K$ if and only if $Z\in K(n+m)$ implies $\al=0$.
 The tuple  $Z$ is  $n$-\df{block diagonalizable}
  provided there exists a $k$, integers $n_j\le n$ and tuples  $E^j\in \mS_{n_j}^g$ such that
  $Z \usim  E^1 \oplus E^2 \oplus  \cdots \oplus E^k.$

Our main general theorem is as follows.

\begin{theorem}
 \label{thm:intromainext}
 Suppose $K$ is a fully free set, $n$ is a positive integer and $X\in K(n)$.

 \begin{enumerate}[\rm (1)]
  \item
   \label{it:Euclidean-geometric}
     $X$ is  a Euclidean extreme point of $K(n)$ if and only if, if  $\al,\be\in \smatng$ {\CR and $\be=X$ or $K$ is a free spectrahedron, see Section \ref{sec:FreeSpecDefs},}   and $Z$ as in equation \eqref{eq:Z} is in $K(2n)$, then  $\al=0$;
  \item
   \label{it:0arv}
     $X$ is a matrix extreme point of $K$ if and only if $X$ is irreducible and for each positive integer $m$, $g$-tuple $\al$ of $n\times m$ matrices and tuple $\be\in \smatmg$  if $Z$ is in $K(n+m)$ and is $n$-block diagonalizable, then $\al=0$;
  \item
   \label{it:wearekleski}
     $X$ is an absolute extreme point of $K$ if and only if  $X$ is irreducible and in the Arveson boundary of $K$.
 \end{enumerate}
\end{theorem}

The definition here of an Arveson boundary point for a free set mirrors the  geometric formulation of a (not necessarily irreducible) boundary representation \cite{Arv69} used by a number of authors including \cite{Aglermodel,MS98,DM05,Arv08}. Item \eqref{it:wearekleski} explicitly connects, in the setting of $\smatg$, Kleski's notion of boundary point for a matrix convex set with that of an Arveson boundary point.

Theorem \ref{thm:intromainext} makes clear the implications, Arveson boundary implies matrix extreme
implies Euclidean extreme.
 Item \eqref{it:Euclidean-geometric} falls out of the usual argument that if $K$ is matrix convex, then each $K(n)$ is convex. It is stated and proved as Proposition \ref{prop:Euclidean-geometric}. Items \eqref{it:0arv} and \eqref{it:wearekleski} are stated and proved as  Theorems \ref{thm:0arv} and \ref{thm:wearekleski} respectively.

\subsection{Free spectrahedra and polar duals}
\label{sec:FreeSpecDefs}
A simple class of matrix convex sets is the
solution sets of linear matrix inequalities (LMIs).
Given a $g$-tuple $A\in\smatg$,  let
$\La_A$  denote the \df{homogeneous linear pencil}
\[
\La_A(x)=A_1 x_1+\cdots+ A_g x_g,
\]
and
$L_A$ the \df{monic linear pencil}
\[
L_A(x)=I - A_1 x_1-\cdots- A_g x_g.
\]
The corresponding \df{free spectrahedron}
$\cD_A= \big(\cD_A(n)\big)_n$ is the sequence of sets
\[
\cD_A(n)=\{ X \in \smatng \mid L_A(X)\succeq0\}.
\]

 It is routine to verify that a free spectrahedron is  matrix convex.  Thus free spectrahedra are matrix convex subsets of $\smatg$ satisfying a finiteness condition.  Conversely,  as a special case of the Effros-Winkler matricial Hahn-Banach Theorem \cite{EW97},  if $K\subset \smatg$ is compact matrix convex and $0\in K(1)$, then $K$ is the (possibly infinite) intersection of free spectrahedra.\looseness=-1

If $\Omega\in \smatdg$ we say $\Omega$ has \df{size} $d$.  A tuple $B$ is a defining tuple for $\cD_A$ if $\cD_B=\cD_A$ and $\Omega$ is a  \df{minimal defining tuple} if $\Omega$ has minimal size among all defining tuples. The Gleichstellensatz
\cite[Corollary 3.18 or Theorem 3.12]{HKM13} (see also \cite{Zal+}) says any two minimal defining tuples for  $\cD_A$ are unitarily equivalent.

By analogy with the classical  notion, the \df{free polar dual} \index{polar dual}
$K^\circ =(K^\circ(n))_n$ of a free set $K\subset\smatg$ is
\[
 K^\circ(n):=  \Big\{ A \in \smatng : \
 L_A(X)=  I \otimes I - \sum_j^g A_j \otimes X_j \succeq 0
\text{ for all } X  \in K \Big\}.
\]
Note that the (free) polar dual of a matrix convex set is closed and  matrix convex.
 We refer the reader  to \cite{EW97,HKMjems} for basic properties of polar duals.

We do not know conditions on a closed matrix convex set  $K$
equivalent to the condition that $K$ is the closed matrix convex hull of its Arveson boundary points.
However, with a finiteness hypothesis Theorem \ref{thm:weak-kleski}, inspired by   \cite[Corollary 2.4]{Kls+},   gives an answer.

\begin{theorem}
\label{thm:weak-kleski}
Suppose $K$ is a closed matrix convex set containing $0$.
  If $K^\circ = \cD_{\Omega}$
  and $\Omega$ is a minimal defining tuple for $\cD_{\Omega}$, then there exists
 an $N$ and irreducible tuples $\Omega^1,\dots,\Omega^N$ in the Arveson boundary of $K$  such that
$
 \Omega=\oplus \Omega^j
$
 and
\beq\label{eq:omegaSpans}
 K=\mco(\{\Omega\})=\mco(\{\Omega^1,\dots,\Omega^N\}).
 \eeq
 (Here $\mco (\Gamma)$ denotes the matrix convex hull of $\Gamma\subset\smatg$, i.e., the smallest matrix convex set containing $\Gamma$.)
 Conversely, if there exists a tuple $\Omega$ such that \eqref{eq:omegaSpans} holds, then $K^\circ = \cD_{\Omega}$.
\end{theorem}

 Theorem \ref{thm:weak-kleski} is proved in Section \ref{sec:polArv}.

 The theory of extreme points can be used to obtain properties
of several basic spectrahedra.
For example, in Section \ref{sec:FreeSimpAndDuals} we use it to deduce the following.

\begin{cor}
 \label{cor:hardtobepolar-intro}
 Let $A\in \smatdg$. If {\CR $\cD_A (2)= \overline{\cD_A (2)}$ and} the polar dual of $\cD_A$ is again a free spectrahedron, then $\cD_A(1)$ is
 a polyhedron. In particular, if $\cD_A(1)$ is a ball, then $\cD_A^\circ$
 is not a free spectrahedron.\looseness=-1
\end{cor}

Corollary \ref{cor:hardtobepolar-intro} implies  finitely generated matrix convex sets are rarely free spectrahedra.  However, they are
\df{free spectrahedrops}, i.e., projections of free spectrahedra \cite{HKMjems}.  On the other hand, if $K$ is a compact matrix convex free semialgebraic set containing $0$ in its interior, then $K$ is a free spectrahedron \cite{HM12}.

\subsection{Reader's Guide}
The rest of the paper is organized as follows.
Section \ref{sec:classyextreme}
recalls the classical notion of an extreme point in $\R^g$
and characterizes those for
matrix convex sets (Proposition \ref{prop:Euclidean-geometric})
and free spectrahedra (Corollary \ref{cor:RG+}).
Section \ref{ssec:arv} gives our first main result,
Theorem \ref{thm:wearekleski} showing that an absolute
extreme point is exactly an irreducible Arveson boundary point.
In Section \ref{sec:block} we prove our second main result, a dilation
theoretic characterization of matrix extreme points via
the \barvb, see Theorem \ref{thm:0arv}.
Section \ref{sec:polArv} gives the proof of our final main result:
the polar dual of a matrix convex set is spanned by its finitely
many Arveson boundary points if and only if it is a free
spectrahedron (Theorem \ref{thm:weak-kleski}).
We show the polar dual of a free spectrahedron is seldom
a free spectrahedron in Corollary \ref{cor:hardtobepolar}, and show it is one for a free
simplex (Corollary \ref{cor:pdSimplex}).
The paper concludes with Section \ref{sec:exA} providing
further applications and examples of matrix extreme points.
The sets of extreme points for two types of
matrix convex sets above the unit
disk in $\R^2$
are studied in Subsection \ref{ssec:ball}, where
the Arveson boundary of the spin disk is identified and shown
to span.  Finally, in Subsection \ref{ssec:TV} the matrix
convex hull of the TV screen is investigated.

\section{Euclidean (Classical) Extreme Points}
 \label{sec:classyextreme}
In this section we establish Theorem \ref{thm:intromainext} \eqref{it:Euclidean-geometric} characterizing  Euclidean
extreme points of fully free sets. This characterization reduces to a result of  Ramana-Goldman
\cite{RG95}  in the case $K$ is a free spectrahedron,
see Corollary \ref{cor:RG+}.

 Recall,  a point  $v$  of a convex set $C\subset \R^g$ is an \df{(Euclidean) extreme point} of $C$,
  in the classical sense,
if $ v= \la a + (1-\la) b$ for $a,b \in C$ and $0 < \la < 1$
implies $a=v=b$.
Let  $\euc C$ denote its set of Euclidean extreme points.
Note that this definition makes sense even if $C$ is not assumed convex.

  While the interest here is in
matrix convex sets $K$ such as
  free spectrahedra, it is
 of course natural to consider, for fixed $n$, the
 extreme points of the
convex set $K(n)\subset\smatng$.
The next result is a dilation style characterization of
Euclidean extreme points of $K(n)$.
\begin{prop}
 \label{prop:Euclidean-geometric}
 Suppose $\Ga$ is a fully free set, $n$ is a positive integer and $X\in\Ga(n).$  If $X$ is a Euclidean extreme point of $\Ga(n)$,
 $\al\in \smatng$ and
\[
 \begin{pmatrix} X & \alpha \\ \alpha & X \end{pmatrix} \in   \Ga(2n) ,
\]
then $\alpha =0$.  
Conversely, if $X$ is not an extreme point, then there exists $0\neq\alpha\in \smatng$ with
\begin{equation}
\label{eq:Wa}
 W=\begin{pmatrix} X & \alpha \\ \alpha & X \end{pmatrix}\in \Ga(2n) .
\end{equation}

Finally, if $\Ga$ is matrix convex, $X$ is a Euclidean extreme point 
and for some  $\al\in\matng$ we have
\begin{equation}
 \label{eq:Xab}
 Z=\begin{pmatrix} X & \al^*\\ \al & X\end{pmatrix} \in \Ga(2n),
\end{equation}
 it follows that $\al=0$.
\end{prop}

\begin{proof}
 Suppose $X\in\Gamma(n)$ is not a Euclidean extreme point for $\Gamma(n)$. Thus, there exists a $0\ne \al\in \smatng$ such that $X\pm\al\in \Gamma(n)$. Since $\Gamma$ is closed with respect to direct sums, $Y=(X+\al)\oplus (X-\al)\in \Gamma(2n)$.
Let $U$ denote the $2n\times 2n$ unitary matrix
\begin{equation}
\label{eq:U}
 U=\frac{1}{\sqrt{2}} \begin{pmatrix} I_n & -I_n\\ I_n & I_n \end{pmatrix}
\end{equation}
 Since $\Gamma$ is closed with respect to unitary similarity,
\[
 U^* Y U = \begin{pmatrix} X&\al\\ \al & X\end{pmatrix} \in \Gamma(2n).
\]

To prove the converse suppose $X\in \Ga(n)$ is a Euclidean extreme point, $\al\in \smatng$ and 
$W$ is as in equation \eqref{eq:Wa}. With $U$ as in equation \eqref{eq:U}, we have
\[
 UWU^* = \begin{pmatrix} X+\al & 0 \\ 0 & X-\al \end{pmatrix} \in \Ga(2n).
\]
 Since $\Ga$ is fully free, both $X\pm\al\in\Ga(n)$.  Thus, $X= \frac12 [(X+\al)+(X-\al)]$. Since $X$ is an extreme point, $X+\al=X$; i.e., $\al=0$.

Finally, suppose $\Ga$ is matrix convex and  $Z$  is as in equation \eqref{eq:Xab}. Let
\[
 F=\begin{pmatrix} 0 & I\\ I & 0 \end{pmatrix}, \  \ G=\begin{pmatrix} iI & 0\\ 0 & I \end{pmatrix}.
\]
Since 
\[
\frac{1}{2}(F^*ZF + Z)  = \begin{pmatrix} X & \gamma \\ \gamma & X\end{pmatrix},
\]
where $\gamma =\frac 12 (\al+\al^*)$, it follows from what has already been proved that $\al+\al^*=0$.
On the other hand,
\[
 \frac 12 (G^* Z G + GF^*ZFG^*) = \begin{pmatrix} X & \delta \\ \delta & X\end{pmatrix}\in \Ga(2n),
\]
where $\delta =\frac 12 (-i\al +i\al^*) = \frac{i}{2}(\al^*-\al)$. Thus, again by what has already been proved,
 $\al-\al^*=0$. Hence $\al=0$.
\end{proof}

We next consider, for fixed $n$, the
 extreme points of the spectrahedron $\cD_{A}(n)\subset \smatng$ and
  this was done by Ramana and Goldman \cite{RG95} for $n=1$.  Below
 is a sample result.

\begin{theorem}[Ramana and Goldman \protect{\cite[Corollary 3]{RG95}}]
 \label{thm:RG}
   Fix $A\in \smatdg$. A point $X\in\cD_{A}(1)$ is an Euclidean extreme point of $\cD_A(1)$
if and only if
   $Y\in\R^g$ and $\ker L_A(X) \subset \ker\La_A(Y)$ implies
   $Y=0$.
\end{theorem}

The next corollary is an immediate consequence of Proposition 
\ref{prop:Euclidean-geometric}
and the proof of Theorem \ref{thm:RG}.

\begin{cor}
\label{cor:RG+}
 For a point $X\in\cD_A(n)$ the following are equivalent.
  \ben[\rm (i)]
   \item \label{it:euclid}
     $X$ is an Euclidean extreme point of $\cD_{A}(n)$;
   \item \label{it:rLE}
         $Y\in\smatng$ and $\ker L_A(X)\subset \ker \La_A(Y)$ implies
             $Y=0$;
   \item
    \label{it:kercond}
     $W\in\cD_A(n)$ and $\ker L_A(X)\subset \ker L_A(W)$ implies
           $\ker L_A(X)=\ker L_A(W)$;
   \item\label{it:rigArv} $\al,\beta\in\smatng$ and
     \[
      Z=\begin{pmatrix} X &\al \\ \al & \be \end{pmatrix} \in \cD_A(2n)
     \]
      implies $\al=0$.
  \een
\end{cor}

\begin{proof}
  The proof uses the following observation. If $L(X)\succeq 0$ and 
$Y\in\smatng\setminus\{0\}$ satisfies
  $\ker L(X) \subset \ker \La(Y)$,
  then there exists a $t\in\R$ such that $L(X\pm tY)\succeq 0$ and
  the kernel of $L(X+tY)$ strictly contains the kernel of $L(X)$.
  To prove this assertion, simply  note that $L(X\pm tY) = L(X)\pm t\La (Y)$ and
  the hypotheses imply that the range of the self-adjoint matrix $\La(Y)$ lies inside the range of
  the positive semidefinite matrix $L(X)$.

 Theorem \ref{thm:RG} gives the equivalence of items \eqref{it:euclid} and \eqref{it:rLE}
 in the case $n=1$, but their proof adapts easily to  general $n$. 
 A proof is included here for the readers convenience.
  If $X$ is an extreme point of $\cD_{A}$,
  and $\ker L_A(X) \subset \La_A(Y)$, then  choosing $t$ as
  above, gives $L_A(X\pm tY) \succeq 0$. Hence $X\pm tY \in \cD_{A}$
  and therefore $Y=0$.  Hence item \eqref{it:euclid} implies item \eqref{it:rLE}.
   Conversely,  if $X$ is not an extreme point of $\cD_{A}$,
 then there exists a non-zero $Y$ such that $X\pm Y \in \cD_{A}$.
  In particular, $0\preceq L_A(X+Y) = L_A(X)\pm \La_A(Y)$
  and hence $\ker \La(Y)\supset \ker L_A(X)$.  Hence item \eqref{it:euclid} and item \eqref{it:rLE}
  are equivalent.

  Suppose $Y\ne 0$ and  $\ker L_A(X)\subset  \La_A(Y)$.
  By the observation at the outset of the proof, there is a $t\ne 0$ such that
 \[
    0\preceq L_A(X \pm  tY)= L_A(X) \pm t \La_A(Y).
 \]
  and the kernel of $L_A(X+tY)$ strictly contains that of $L_A(X)$.
  Hence item \eqref{it:kercond} implies item \eqref{it:rLE}.
  On the other hand, if $\ker L_A(X)\subset \ker L_A(W)$, then with
 $Y= X-W$, the kernel of $L_A(X)$ is contained in the kernel of $\La_A(Y)$.
  In particular, if (ii) holds, then $Y=0$ and hence $W=X$.
   Thus $\ker L_A(X)= \ker L_A(W)$ and item \eqref{it:rLE} implies item \eqref{it:kercond}.

That item \eqref{it:rigArv} implies item \eqref{it:euclid} 
is seen by taking $\be=X$ and applying   Proposition \ref{prop:Euclidean-geometric}.
Now suppose item \eqref{it:rLE} holds and $Z$ is given as in item \eqref{it:rigArv}. By considering $L_A(Z)$, it follows
that $\ker(L_A(X))\subset \ker(\La_A(\alpha))$. Hence $\al=0$. 
\end{proof}



\section{Arveson Boundary and Absolute Extreme Points}\label{ssec:arv}

Now we turn to Arveson boundary points and
 absolute extreme points
({\it boundary points} in the terminology of \cite{Kls+})
of a free set as defined in the introduction.
We shall establish Theorem \ref{thm:intromainext} \eqref{it:wearekleski}, see  Theorem \ref{thm:wearekleski},
showing that  absolute extreme points are exactly irreducible Arveson boundary points.

\subsection{Matrix convex hulls}\label{sec:mco}
In this subsection we describe the matrix convex hull of an arbitrary subset $\Gamma\subset\smatg$.

The
\df{matrix convex hull} of an arbitrary
subset $\Gamma\subset\smatg$
is defined to be the
intersection
of all matrix convex sets
containing $\Gamma$. It is easily seen to be convex and is thus the smallest matrix convex set
that contains $\Gamma$.  Let $K=\mco(\Gamma)$ denote the matrix convex hull of $\Gamma$ and $\cmco(\Gamma)$ the \df{closed matrix convex hull} of $\Gamma$ obtained by taking the ordinary (in $\smatng$) closures of each $K(n)$. A routine exercise shows $\mco(\Gamma)$ is matrix convex and is thus the smallest closed matrix convex set containing $\Gamma$.   As an example, each $\Omega \in \smatg$ gives rise to the \df{finitely generated} matrix convex set,
\beq\label{eq:mco}
 \mco(\{\Omega\}) =\big\{V^* (I_m\otimes \Omega) V: m\in\N,\,V \text{ is an isometry}\big\}.
\eeq

\begin{prop}
 \label{prop:hull-abstract}
   A point $X\in\smatng$ is in the matrix convex hull of the free set  $\Gamma$ if and only if there is a
   positive integer $N$, a tuple $Z\in \Gamma(N)$ and an isometry $V:\C^n\to \C^N$ such that $X=V^* ZV.$
\end{prop}

\begin{proof}
For positive integers $n$, let   $K(n)$ denote those $n\times n$ tuples $X$ for which there exits a positive integer $N$, an isometry $V:\C^n\to\C^N$ and a tuple $Z\in \Gamma(N)$ such that $X=V^*ZV$.  Since $\Gamma$ is closed with respect to direct sums, so is $K=(K(n))_n$.
 Suppose now $W:\C^\ell\to\C^n$ is an isometry. Let $Y$ denote the isometry $Y:\C^\ell\to \C^{N}$ given by $Y= VW$. It follows that
\[
Y^* Z Y = W^* XW.
\]
 Since $Z\in\Gamma(N)$, by definition $W^*XW\in K(\ell)$. Hence $K$ is closed with respect to isometric conjugation and is therefore a matrix convex set.
\end{proof}

\begin{remark}\rm
 \label{rem:stoopid}
 If $\Gamma\subset \smatg$ is a finite set, then Theorem \ref{thm:weak-kleski} implies $\mco(\Gamma)$ is closed. (It is not hard to give a direct proof of this fact.)
   At least as far as we are aware for a compact set $\Gamma\subset\smatng$, its matrix convex hull $\mco(\Gamma)$ is not necessarily
 closed. Thus $\cmco(\Gamma)$ is potentially larger than $\mco(\Gamma)$.
\end{remark} %

\subsection{Arveson boundary}
In this subsection we recall the notion of the Arveson boundary of a free set
and develop some of its basic properties.

 Given a free set $\Gamma$ and  a positive integer  $n$,  a  tuple $X\in \Gamma(n)$ is in the
 the \df{Arveson boundary of  $\Gamma$} or is an \df{Arveson boundary point} of $\Gamma$,
 written  $X\in\arv \Gamma$, \index{$\arv$}  if given a positive integer $m$ and $g$-tuples $\alpha$
 of size $n\times m$ and $\beta$ of size $m\times m$ such that
\beq\label{eq:defnArv}
   \begin{pmatrix} X & \alpha \\ \alpha^* & \beta \end{pmatrix} \in \Gamma(n+m)
\eeq
 it follows that  $\alpha=0$.  A coordinate free definition is as follows. The point $\Gamma\in K(n)$ is an Arveson boundary point of $K$ if for each $m$ each  $Y$ in $K(m)$ and isometry $V:\C^n\to \C^m$ such that $X= V^*Y V$
 it follows that $VX=YV$.

The following statement is a simple consequence of
Proposition \ref{prop:Euclidean-geometric}.

\begin{prop}
\label{prop:ArvboundaryEuclidean}
 Suppose $\Gamma$ is a free set and $n$ is a positive integer. If $X\in\Gamma(n)$ is an Arveson boundary point for $\Gamma$, then $X$ is a Euclidean extreme point of $\Gamma(n)$.
\end{prop}

The next lemma says for a matrix convex set  the property
\eqref{eq:defnArv} in the definition of an Arveson boundary point only needs to be verified for column tuples $\alpha$.

\begin{lemma}
 \label{lem:reduce}
   Suppose $K$ is matrix convex and
\[
  Z = \begin{pmatrix} X & \al\\ \al^* & \be\end{pmatrix} \in K(m+\ell).
\]
 If $v\in \mathbb R^{\ell}$ is a unit vector, then
\[
  Y =\begin{pmatrix} X & \al v \\ v^*\al^* &  v^*\be v
  \end{pmatrix}  \in K(m+1).
\]
\end{lemma}

\begin{proof}
 Consider the isometry
\[
  W= \begin{pmatrix}  I_n & 0 \\ 0 & v\end{pmatrix}.
\]
 Since $K$ is matrix convex,
\[
  W^* Z W = Y \in K(m+1).\qedhere
\]
\end{proof}

\begin{lemma}
 \label{lem:scalar enough for arv}
  Suppose $\Gamma\subset \smatg$ is matrix convex.  A tuple  $X\in \smatng$ is in the Arveson boundary
  of  $\Gamma$ if and only if, given a $g$-tuple $\al$  from $\C^n$
  and $\be\in \C^g$ such that
\[
   \begin{pmatrix} X & \al \\ \al^* & \be \end{pmatrix} \in \Gamma
\]
 it follows that $\al=0$.
\end{lemma}

\begin{proof}
Apply Lemma \ref{lem:reduce}.
\end{proof}

\begin{lemma}
\label{lem:boundaryisfree}
If $\Gamma\subset \smatg$ is a free set, then the Arveson boundary of $\Gamma$ is closed with respect to  direct sums and unitary similarity; i.e., it is a free set.
\end{lemma}

\bep
The proof regarding  unitary conjugation is trivial. Onward to direct sums.
Suppose $X,Y$ are both in the Arveson boundary  and
\beq
Z:=\bem
X & 0 &                          B_{11} & B_{12}\\
0 & Y &                          B_{21} & B_{22} \\
 B_{11} & B_{21}^* &    D_{11} & D_{12} \\
  B_{12}^* & B_{22}^*   &  D_{12}^* & D_{22}
\eem
\eeq
is in  $ \Gamma$.  Writing $Z$ as
\[
 \begin{pmatrix} X & \begin{pmatrix} 0 & B_{11}&B_{12} \end{pmatrix} \\ \begin{pmatrix} 0 \\ B_{11}^* \\ B_{12}^* \end{pmatrix}  &
    \begin{pmatrix} Y &  B_{21} & B_{22} \\
      B_{21}^* &    D_{11} & D_{12} \\
  B_{22}^*   &  D_{12}^* & D_{22}\end{pmatrix} \end{pmatrix} \in \Gamma
\]
and using the assumption that $X\in \arv\Gamma$ it follows that $B_{11}=0$
and $B_{12}=0$. Reversing the roles of $X$ and $Y$ shows $B_{21}=0$ and $B_{22}=0$.
\eep

\begin{prop}
 \label{lem:SEinCTV}
  If  $\Gamma$ is a fully free set, 
  then $\arv\Gamma = \arv\mco\Gamma.$
\end{prop}

\begin{lemma}
 \label{lem:preCTV}
   If $\Gamma\subset\Gamma^\prime \subset \smatg$ and $X\in\Gamma$ is an Arveson boundary point of $\Gamma^\prime$, then $X$ is an Arveson boundary point of $\Gamma$.
\end{lemma}

\begin{proof}
 Evident.
\end{proof}

\begin{proof}[Proof of Proposition \ref{lem:SEinCTV}]
 Suppose $A$ is in $\arv\mco\Gamma.$  
 Since $A$ is in $\mco\Gamma$, by Proposition \ref{prop:hull-abstract}
  there exists an $X$ in $\Gamma$ of the form
$$
X=\bem A & B     \\
B^*  &C
\eem.
$$
Since $X\in\mco\Gamma$ and  $A$ is an
Arveson boundary point of $\mco\Gamma$,
 it follows that $B=0$.
 Since $\Gamma$ is closed with respect to restrictions
  to reducing subspaces, $A\in\Gamma$.   By Lemma \ref{lem:preCTV}, $A\in \arv\Gamma$.

For the converse, suppose
$A\in\arv\Gamma$ and let $\alpha,W$ such that
$\al\neq0$ such that
\[
 Z=
\bem
A&\al \\ \al^*&W\eem \in\mco\Gamma
\]
 be given.  By Proposition \ref{prop:hull-abstract}, there
is a dilation $X$ of $Z$ lying in $\Gamma$; i.e., there exists $\de, \ga, Y$ such that
\begin{equation}
 \label{eq:dilateA}
X=
\bem
A &\al & \ga  \\ \al^*&W & \de \\
\ga^* & \de^* & Y
\eem \in\Gamma.
\end{equation}
 Since $A\in \arv\Gamma$ and $X\in\Gamma$, it follows that $\alpha=0$. Hence $A\in \arv\mco(\Gamma)$.
\end{proof}

\begin{remark}\rm
 \label{rem:stoopid+}
   Assuming $\Gamma$ is compact, we do not know if Proposition \ref{lem:SEinCTV} holds with $\arv\mco\Gamma$ replaced by $\arv\cmco\Gamma$ a statement equivalent to
$\arv\mco\Gamma=\arv\cmco\Gamma.$
\end{remark}

\subsection{Absolute extreme points vs.~\arvb}\label{sec:absE}
  The tuple $X\in\smatg$ is \df{irreducible} if there does not
  exist a nontrivial common invariant subspace for
  the set $\{X_1,\dots,X_g\}.$ Observe  invariant equals reducing
  since the matrices $X_j$ are self-adjoint.

\begin{thm}
 \label{thm:wearekleski}
   Suppose $K$ is a fully free set.   A point $X\in K$ is an absolute extreme point of $K$ if and only if $X$ is irreducible and in the Arveson boundary of $K$.
\end{thm}

The fact that an absolute extreme point is irreducible is rather routine. For completeness, we include a proof.

\begin{lemma}
 \label{lem:irreduc}
 Suppose $K\subset \smatg$ is a fully free set.
  If $X\in K$ is either a matrix extreme point or an absolute extreme point of $K$, then $X$ is irreducible.
\end{lemma}

\begin{proof}
  Suppose $X\in K(n)$ is a matrix extreme point of $K$.
   If $X$ is not irreducible, then
  there exists nontrivial reducible subspaces $H_1$ and $H_2$ for the set $\{X_1,\dots,X_g\}$
  such that $\mathbb C^n = H_1\oplus H_2$.
   Let $V_\ell:H_\ell\to \mathbb C^n$
  denote the inclusion mappings. Let $Y^\ell= V_\ell V_\ell^* X V_\ell V_\ell^*$.  Since $K$ is closed with respect to reducing subspaces, $Y^\ell\in K$ and moreover the mappings $V_\ell V_\ell^*$ are proper.   Evidently,
\[
  X = \sum V_\ell^* Y^\ell V_\ell.
\]
 Since this sum  is a  proper combination from $K$, we have that $X$ and $Y^\ell$ are unitarily equivalent, a contradiction since the size of $X$ strictly exceeds that of $Y_\ell$.  If $X$ is an absolute extreme point, then $X$ is a matrix extreme point and hence is irreducible by what has already been proved.
\end{proof}

\subsubsection{A non-interlacing property}
In this subsection we present a lemma on
an interlacing property we will use in the proofs of
Theorems \ref{thm:wearekleski} and \ref{thm:weak-kleski};
cf.~Cauchy's interlacing theorem \cite[Theorem 4.3.8]{HJ12}.

\begin{lemma}\label{lem:interlace}
Let  $D$ denote the $n\times n$  diagonal matrix with diagonal  entries $\lambda_1,\dots,\lambda_n\in\C$ and let
\[
  E = \begin{pmatrix} D & a\\ a^* &  e\end{pmatrix}.
\]
 If $f\in \C$ and  $E$ is unitarily equivalent to
\[
  F=\begin{pmatrix} D&0\\0 & f\end{pmatrix},
\]
then $a=0$.
\end{lemma}

\begin{proof}
  Denote the entries of $a$ by $a_1,\dots,a_n$. By induction,
\[
 \det(E-tI) = p(t)(e-t) - \sum a_j^2 \prod_{k\ne j} (\lambda_k-t),
\]
 where $p(t) =\prod (\lambda_k-t)$.
 On the other hand, $\det(F-tI)=p(t)(f-t)$. Thus,
\[
 p(t)(e-f) = p(t)(e-t) -p(t)(f-t) = \sum a_j^2 \prod_{k\ne j} (\lambda_k-t).
\]
 The left hand side is $0$ or has degree $n$ in $t$; whereas the right hand
 side is either $0$ or has degree $n-1$ in $t$. The conclusion is that both
 sides are $0$ and thus $a_j=0$ and $e=f$.   Here we have used in a very
  strong way that the matrices $E$ and $F$ are each self-adjoint
 (in which case unitary equivalence is the same as similarity).
 \end{proof}

\begin{lemma}
 \label{lem:irrconsequence}
  Suppose $K\subset \smatg$ is free set and $n$ is a positive integer.
If $X\in K(n)$ is an absolute extreme point of $K$ and  if $\al$ is a g-tuple from $\C^n$, $\be\in \R^g$ and
\[
  Y = \begin{pmatrix} X & \al \\ \al^* & \be \end{pmatrix}
  \in K(n+1)
\]
 then $\al=0$.

\end{lemma}

\begin{proof}
 Letting $V:\C^n\to\C^{n+1}=\C^n\oplus \C$ denote the isometry $Vx= x\oplus 0$, observe $X=V^*YV$. Using the
 definition of absolute extreme point, it follows that there is a $\ga\in \C$ such that
\[
 Y\usim Z= \begin{pmatrix} X&0\\0 &\ga\end{pmatrix}.
\]
By Lemma \ref{lem:interlace}, $\al=0$ and the proof is complete.
\end{proof}

\subsection{Proof of Theorem \ref{thm:wearekleski}}

For the proof of the forward direction of  Theorem {\rm\ref{thm:wearekleski}},
suppose $X$ is an absolute extreme point.
Lemmas \ref{lem:irreduc} and 
 \ref{lem:irrconsequence} together with Lemma \ref{lem:scalar enough for arv} show $X$ is irreducible and an Arveson boundary point respectively.

 The proof of the reverse direction of Theorem {\rm\ref{thm:wearekleski}} uses the following lemma.

\begin{lemma}
 \label{lem:irreduciblegives}
Fix positive integer $n$ and $m$  and suppose  $C$ is a nonzero $m\times n$ matrix,  the tuple $X\in\smatng$  is irreducible  and $E\in \smatmg$.
If  $CX_i = E_i C$ for each $i,$  then $C^*C$
 is a nonzero multiple of
 the identity. Moreover, the range of
 $C$ reduces the set $\{E_1,\dots,E_g\}$ so that, up to unitary
 equivalence,  $E= X\oplus Z$ for some $Z\in \mathbb S_k^g,$ where $k=m-n$.
\end{lemma}

\begin{proof}
To prove this statement note that
\[
   X_j C^* = C^* E_j.
\]
 It follows that
\[
 X_j C^*C = C^* E_j C = C^* C X_j.
\]
 Since $\{X_1,\dots,X_g\}$ is irreducible, $C^*C$ is a nonzero multiple
 of the identity and therefore  $C$ is a multiple of an isometry.
 Further, since $CX=EC$, the range of $C$ is invariant for $E$.  Since each $E_j$
 is self-adjoint, the range of $C$ reduces each $E_j$ and $C$, as a mapping
 into its range is a multiple of a unitary.
\end{proof}

 To complete the proof of the reverse direction,  suppose $X$ is both irreducible and in the Arveson boundary of $K$.  To
  prove $X$ is an absolute extreme point, suppose
\[
   X= \sum_{j=1}^\nu C_j^* E^j C_j,
\]
 where each $C_j$ is nonzero, $\sum_{j=1}^\nu C_j^* C_j =I$ and $E_j\in K$. In this case, let
\[
   C= \begin{pmatrix} C_1 \\ \vdots \\ C_\nu \end{pmatrix}
\quad\text{ and }\quad
  E=
E^1\oplus E^2\oplus\cdots\oplus E^\nu
\]
 and observe that $C$ is an isometry and $X=C^*EC$.
   Hence, as $X$ is in the Arveson boundary,
  $CX=EC$. It follows that $C_j X_k = E^j_k C_j$ for each $j$ and $k$.  Thus, by Lemma \ref{lem:irreduciblegives},
  it follows that each $E_j$ is unitarily equivalent to $X\oplus Z^j$
  for some $Z^j\in K$.
Thus $X$ is an 
absolute
extreme point.
\hfill\qedsymbol

\section{Matrix Extreme Points and the \BarvB}\label{sec:block}

In this section we turn our attention to
 matrix extreme points and the \barvb.
 Theorem \ref{thm:0arv} is the main result of this section. It
 gives an Arveson-type dilation characterization
 of matrix extreme points.

 We say the tuple
 $X \in K(n)$
 is in  the \df{\barvb} or is a \df{block-Arveson boundary point} of a matrix convex set $K$ if
 whenever
 $$  \begin{pmatrix} X & \al \\ \al^* & \be \end{pmatrix} \in K $$
 is $n$-block diagonalizable, $\al  =0$.

\begin{thm}\label{thm:0arv}
  Let $K$ be a  matrix convex set.
  A point $X$ is a matrix extreme point of $K$ if and only
  if $X$ is both irreducible and in the \barvb\ of $K$.
\end{thm}

\subsection{Matrix extreme points}
\label{ssec:matextpts}

We  now
recall the definition of matrix extreme points
and the corresponding  Krein-Milman theorem of Webster-Winkler \cite{WW99}.

 The matrix convex combination \eqref{eq:SumY}
is \df{proper} provided each $V_\ell$ is surjective.  Note that in this case, $n\ge n_\ell$ for each $\ell$. A tuple  $X \in K(n)$  is a \df{matrix extreme point}
  of the matrix convex set $K$
 if whenever it is represented as a proper matrix combination of the form
\eqref{eq:SumY} with $Y^\ell \in K(n_\ell)$,
then $n=n_\ell$ and  $X \usim  Y_\ell$ for each $\ell$. (Recall: we use $A\usim B$ to denote $A$ and $B$ are unitarily equivalent.)
The set of all matrix extreme points will be denoted by $\cEm K $\index{$\cEm K$}.
Webster-Winkler give in \cite[Theorem 4.3]{WW99} a Krein-Milman theorem in the matrix convex setting. (See also \cite{F00} for an alternate proof and the elucidating discussion of the connection between matrix extreme points of $K$ and matrix states on an operator system.)

\begin{theorem}[Webster-Winkler \protect{\cite[Theorem 4.3]{WW99}}]
 \label{thm:ww}
If $K\subset\smatg$ be a compact matrix convex set,
then $\cEm K $  is non-empty, and
$K$ is the smallest closed matrix convex set
containing $\cEm K $, i.e.,
 \[ \cmco  \cEm K =K.\]
\end{theorem}

\begin{cor}\label{cor:WWbarv}
 Let $K$ be a compact matrix convex set.
 Then $K$ is the smallest closed matrix
 convex set containing the \barvb\ of $K$.
 \end{cor}

\begin{proof}
Simply combine Theorem \ref{thm:0arv} with Theorem \ref{thm:ww}.
\end{proof}

\begin{remark}\rm
 \label{rem:proper}
 Suppose the matrix convex combination \eqref{eq:SumY} is not necessarily proper. Let $P_\ell$ denote the inclusion of the range of $Y^\ell$ into $\C^{n_\ell}$  and let $Z^\ell = P_\ell^* Y^\ell P_\ell.$ Likewise, let $W_\ell = P_\ell^* V_\ell$.  With these choices,
\[
 X= \sum_{\ell} W_\ell^* Z^\ell W_\ell
\]
 is a proper matrix convex combination. Moreover, the set $K$ is closed with respect to isometric similarity and if each $Y^\ell\in K(n_\ell)$, then this sum is a matrix convex combination from $K$.
\end{remark}

The following lemma gives a convenient alternate definition of matrix extreme point in the setting of $\smatg$.

\begin{lemma}
 \label{lem:behave}
   Suppose $K$ is a matrix convex set, $n$ is a positive integer and $X\in K(n)$. The point $X$  is a matrix extreme point of $K$ if and only if whenever $X$ is represented as a convex combination as in equation \eqref{eq:SumY} with  $n_\ell\le n$ and $Y^\ell \in K(n_\ell)$ for each $\ell$, then, for each $\ell,$ either $V_\ell=0$ or $n_\ell=n$ and $Y^\ell \usim X$.
\end{lemma}

\begin{proof}
 First suppose $X$ is a matrix extreme point of $K$ and is represented as in equation \eqref{eq:SumY} with  $n_\ell\le n$ and $Y^\ell \in K(n_\ell)$ for each $\ell$. In the notation of Remark \ref{rem:proper}, $Z^\ell \in K$ by matrix convexity and each $W_\ell$ is proper (or $0$). Let $L=\{\ell : W_\ell\ne 0\}$.  Thus,
\[
 X= \sum_{\ell\in L} W_\ell^* Z^\ell W_\ell
\]
 is a proper convex combination. Since $X$ is matrix extreme $Z^\ell \usim X$ and therefore the range of $W_\ell$ (equal the range of $V_\ell$) is $n$. Since $n_\ell\le n$, it follows that for each $\ell$ either $V_\ell =0$ or $V_\ell$ is proper and  $Y^\ell \usim X$.

 To  prove the converse fix $X\in K(n)$ and suppose whenever  $X$ is represented in as a matrix convex combination as in  equation \eqref{eq:SumY} with  $n_\ell\le n$ and $Y^\ell \in K(n_\ell)$, then either   $V^\ell=0$ or $n_\ell =n$ and $Y^\ell\usim X$.  To prove $X$ is a matrix extreme point of $K$, suppose  $X$ is represented as a matrix convex combination as in equation \eqref{eq:SumY} and each $V_\ell$ is surjective. In particular, $n_\ell \le n$. Since $V_\ell \ne 0$, it follows that $Y^\ell\usim X$ for each $\ell$ and therefore $X$ is a matrix extreme point of $K$.
\end{proof}

\subsection{Points which are convex combinations of themselves}

Before turning to the proof of Theorem \ref{thm:0arv}, we present an intermediate result of independent interest.

\begin{prop}
 \label{lem:hasagap}
   Suppose $X$ is a tuple of size $N$ and $C_1,\dots,C_t$ are $N\times N$ matrices. If
 \begin{enumerate}[\rm (i)]
  \item
   \label{it:extremegap}  $X\in\smatng$ is irreducible;
  \item
   \label{it:extremeproxy}
   $ \displaystyle
   X = \sum_{\ell=1}^t  C_\ell^* X C_\ell$;
 \item  $\sum C_\ell^* C_\ell =I$,
\end{enumerate}
   then each $C_\ell$ is a scalar multiple of the identity.
\end{prop}

\begin{proof}
Consider the completely positive map $\phi:M_N(\C)\to M_N(\C)$ given by
\[A\mapsto \sum_{j=1}^t C_j^* AC_j.
\]
Since $X_j$ are fixed points of $\phi$ and they generate $M_N(\C)$ by
irreducibility,
Arveson's \cite[Boundary Theorem 2.1.1]{Ar72} (cf.~\cite[Remark 2, p.~288]{Ar72}) implies that $\phi$ is the identity map.
The Unitary freedom in a Choi-Kraus representation (see e.g.~\cite[Theorem 8.2]{NC10}) now concludes the proof.
\end{proof}

We next give an alternative proof of
Proposition \ref{lem:hasagap}
based on the theory of linear matrix inequalities
and free spectrahedra.

\begin{lemma}
 \label{lem:pregapmat}
    Let $X\in\mathbb S_N^g$ be  given.
   For $S\in \mathbb S_N^g$, let
    $\Gamma_S$ denote the eigenspace
   corresponding to the largest eigenvalue $\lambda_S$ of $\Lambda_X(S)$. \index{$\lambda_S$}
   
 If $X$ is irreducible and if $C$ is an $N\times N$  matrix such that
  $(C\otimes I_N) \Gamma_S\subset\Gamma_S$ for all $S\in \mathbb S_N^g$,  then  $C$
   is a multiple of the identity.
\end{lemma}

\begin{lemma}
\label{lem:contain}
 Suppose $X,Y\in \smatg$ and $n\in\mathbb N$.   If $\cD_X(n)\subset \cD_Y(n)$ and $\partial \cD_X(n)\subset \partial \cD_Y(n)$, then $\cD_X(n)=\cD_Y(n)$. 

 If $N\ge M,$ $X\in \mathbb S_N^g$  and $Y\in\mathbb S_M^g$ and if $\cD_X(N)=\cD_Y(N)$, then $\cD_X=\cD_Y$.
\end{lemma}

\begin{proof}
 Suppose $T\in \cD_Y(n)\setminus \cD_X(n)$. Since $\cD_X(n)$ is  convex and contains $0$ in its interior, there is a $0<t<1$ such that $tT\in \partial \cD_X(n)$. Hence $tT\in \partial \cD_Y(n)$, a contradiction (because $\cD_Y(n)$ is also convex and contains $0$ in its interior).

 Now suppose $N\ge M,$ $X\in\mathbb S_N^g$  and $Y\in\mathbb S_M^g$ and  $\cD_X(N)=\cD_Y(N)$.  First observe, if $m\le N$, then $\cD_X(m)=\cD_Y(m)$. Now  let $n\in\N$, $S\in \cD_X(n)$ and a vector $\gamma \in \C^M\otimes \C^n$ be given. Write
\[
 \gamma =\sum_{j=1}^M e_j \otimes \gamma_j \in \C^M\otimes \C^n.
\]
 Let $\cW$ denote the span of $\{\gamma_j:1\le j\le M\}$ and let $W:\cW\to \C^n$ denote the inclusion. Finally, let $T=W^*SW$.  It follows that the size $m$  of $T$ is at most $N$ and $T\in \cD_X(m)$. Hence $T\in \cD_Y(m)$ and therefore,
\[
 0\le  \langle (I-\Lambda_Y(T))\gamma,\gamma\rangle = \langle (I-\Lambda_Y(S))\gamma,\gamma\rangle .
\]
 So $S\in \cD_Y(n)$.  Thus $\cD_X\subset \cD_Y$ and thus, by symmetry, $\cD_X=\cD_Y$.
\end{proof}

\begin{proof}[Proof of Lemma \ref{lem:pregapmat}]
    Since $\Gamma_S$  is invariant under $C\otimes I$, there is a an eigenvector  $\gamma\in\Gamma_S$
   for $C\otimes I.$ (Such eigenvectors exist and we note
  that this is the only place where complex, as opposed to real,  scalars are used.) In particular, letting $\Delta_j$ denote the eigenspaces for $C$,  for each $S$ in  $\mathbb S_N^g$
  there is a $j$ such that $(\Delta_j\otimes \mathbb C^N) \cap \Gamma_S\ne (0)$.

Let $V_j:\Delta_j\to \C^N$ denote the inclusion and let $Y^j = V_j^* X V_j$. Let $Y=\oplus_k Y^j$ (the orthogonal direct sum, even though the subspaces $\Delta_j$ might not be orthogonal).  In particular, $Y\in\mathbb S_M^g$ for some $M\le N$.  Suppose $S\in \cD_X(N)$. It follows that $S\in \cD_{Y^j}$ for each $j$ and hence $X\in \cD_Y(N)$; i.e., $\cD_X(N)\subset \cD_Y(N)$.  If $S\in \partial \cD_X(N),$ then  there is a $j$ and vector $0\ne \gamma_j\in \Delta_j\otimes \C^N$ such that $(I-\Lambda_X(S))\gamma =0$. Let $\gamma^\prime$ denote the corresponding vector in $\oplus_k \Delta_k$; i.e., $0\ne \gamma^\prime = \oplus_k\gamma_k^\prime$ where $\gamma^\prime_k=0$ for $k\ne j$ and $\gamma^\prime_j=\gamma\ne 0$. It follows that
\[
 (I-\Lambda_Y(S))\gamma^\prime = \oplus_k (I-\Lambda_{Y^k}(S))\gamma_k^\prime = (I-\Lambda_{Y^j}(S))\gamma_j^\prime = (I-\Lambda_X(S))\gamma =0.
\]
Hence, $(I-\Lambda_{Y}(S))\gamma^\prime=0$ and therefore $S\in \partial \cD_Y(N)$. Another application of Lemma \ref{lem:contain} now implies $\cD_X(N)=\cD_Y(N)$.  Lemma \ref{lem:contain} implies $\cD_X=\cD_Y$.   Assuming $X$ is irreducible, it is a minimal defining tuple, cf.~\cite[Proposition 3.17 or Corollary 3.18]{HKM13} (see also \cite{Zal+}). On the other hand, the size of $Y$ is at most that of $X$  and thus, by  the Gleichstellensatz
\cite[Theorem 3.12]{HKM13}, they (have the same size and) are unitarily equivalent. Since $X$ is irreducible so is $Y.$  Thus, as $Y=\oplus_k Y^k$, we conclude that there is only one summand, say $Y^1$. Moreover, $C$ has only  one eigenspace $\Delta_1=\C^N$, so it is a multiple of the identity. 
\end{proof}

\begin{proof}[Alternate Proof of Proposition \ref{lem:hasagap}]
  Let  $S\in\mathbb S_N^g$ be given
  and fix $\gamma\in\Gamma_S$. As before, let $\lambda_S$ denote the largest eigenvalue of $\Lambda_X(S)$ and let $\Gamma_S$ denote the corresponding eigenspace.   Note that
\[
  \lambda_S - \Lambda_X(S) = \sum_{\ell=1}^t  (C_\ell^* \otimes I) \big(\lambda_S - \Lambda_X(S)\big) (C_\ell\otimes I).
\]
  Observe   $\lambda_S-\Lambda_X(S)\succeq 0$ and  $(\lambda_S-\Lambda_X(S))\gamma=0$
  together  imply that for each $1\le \ell \le t $,
\[
  (C_\ell\otimes I)\gamma \in\Gamma_S.
\]
Thus $(C_\ell \otimes I_N) \Gamma_S\subset \Gamma_S$ so that
  the subspaces $\Gamma_S$ are all invariant under each $C_\ell \otimes I_N$.
  By Lemma \ref{lem:pregapmat}  each $C_\ell$ is a multiple of the identity.
\end{proof}

\subsection{Proof of Theorem \ref{thm:0arv}}
We are finally ready to prove Theorem \ref{thm:0arv}.
We isolate each of its two implications in a separate lemma.

\begin{lemma}
 \label{lem:extremeimplies0}
   Let $X$ be a matrix tuple of size $N$. If $X$ is matrix extreme in the  matrix convex set $K$,
  then $X$ is in the \barvb\ of $K$.
\end{lemma}

\begin{proof}
 Suppose
\[
  Z=\begin{pmatrix} X & \alpha \\ \alpha^* & \beta \end{pmatrix}\in K
\]
 is $N$-block diagonalizable.   Hence there exists a $t$
  and tuples $E^1,\dots,E^t$ in $K$ of size at most $N$
 and a unitary $U$ such that $Z=U^* E U$, where
\begin{equation}
 \label{eq:ZE}
    E=\oplus_{\ell=1}^t  E^\ell \in K.
\end{equation}
  Letting $C_1,\dots,C_t$ denote the block entries of
  the first column of $U$ with respect to the direct sum decomposition of $E$,
  it  follows that
\[
  X= \sum C_\ell^* E^\ell C_\ell
 \]
  and of course $\sum C_\ell^* C_\ell =I$.
  By Lemma \ref{lem:behave}, since $X$ is matrix extreme and the size of $E^\ell$ is at most $N$,  for each $\ell$ either
  $C_\ell=0$ or $E^\ell$ is unitarily equivalent to $X$; without loss of
  generality we may assume that either $E^\ell = X$ or $C_\ell =0$.
 Let $\mathcal J$ denote the set of indices $\ell$ for which $E^\ell =X$.
  In particular, without loss of generality,
\begin{equation}
 \label{eq:sumcec}
   X= \sum_{\ell} C_\ell^* X C_\ell.
\end{equation}
  From Proposition \ref{lem:hasagap}, each $C_\ell$ is a scalar multiple of the identity.

 To see that $\alpha=0$, write $U$ as a block $t\times 2$ matrix
  whose entries are compatible with $Z$ and the decomposition of $E$
 in  \eqref{eq:ZE}. In particular, $U_{\ell,1}=C_\ell$
  is a multiple of the identity.   Observe that
\[
 \begin{split}
  \alpha_k =  \sum_{\ell} U_{\ell,1}^* X U_{\ell,2}
   =  X \sum U_{\ell,1}^* U_{\ell,2}= 0,
 \end{split}
\]
  and thus  $\alpha=0$.
  Hence $X$ is in the \barvb\ of $K$.
\end{proof}

\begin{lemma}
The  matrix extreme points
 of a matrix convex set $K$ contain the irreducible points in its \barvb.
\end{lemma}

\begin{proof}
  Suppose $X\in K(n)$ is both irreducible and in the \barvb\ of $K$.
 To
  prove $X$ is a matrix extreme point, suppose
\[
   X= \sum_{j=1}^t C_j^* E^j C_j
\]
 where each $C_j$ is nonzero, $\sum_{j=1}^t C_j^* C_j =I$ and $E^j \in K(n_j)$
  for some $n_j\le n.$
 In this case, let
\[
   C= \begin{pmatrix} C_1 \\ \vdots \\ C_t \end{pmatrix}
\qquad
 \text{and}
\qquad
  E=
E^1\oplus \cdots \oplus E^t
\]
 and observe that $C$ is an isometry and $X=C^*EC$. Since $X$ is in the \barvb\ of $K$, it
 follows that $CX=EC$.
  It follows that $C_\ell X_k = E^\ell_k C_\ell$ for each $\ell$ and $k$.
  With $\ell$ fixed, 
  an application of  Lemma \ref{lem:irreduciblegives}
    implies  $E^\ell$ is unitarily equivalent to $X\oplus Z^\ell$
  for some $Z^\ell\in K$. On the other hand, the size of
 $E^\ell$ is no larger than the size of $X$ and hence $Z^\ell=0$. Thus
  $E^\ell$ is unitarily equivalent to $X$.  An  application of  Lemma \ref{lem:behave}  completes the proof.
 \end{proof}

\section{Proof of Theorem \ref{thm:weak-kleski}}\label{sec:polArv}
 The converse part of Theorem \ref{thm:weak-kleski} appears  as \cite[Theorem 4.6]{HKMjems}. For the reader's convenience we provide the short proof here.
  Suppose $X\in (\mco\{\Omega\})^\circ$.  In particular, $L_\Omega(X)\succeq 0$ and hence $X\in \cD_{\Omega}$. Conversely, suppose $X\in \cD_{\Omega}$.  Let $Y\in \mco(\{\Omega\})$ be given. By  \eqref{eq:mco},  $Y= V^* (I_\mu \otimes \Omega)V$. Hence,
\begin{equation}
 \label{eq:shuffle police}
\begin{split}
 L_X(Y) & =  I - \sum_{j=1}^g X_j\otimes Y_j
  =  (I\otimes V)^* \left (I - \sum_{j=1}^g X_j \otimes [I_\mu\otimes \Omega_j]\right ) (I\otimes V)\\
  & \usim  (I\otimes V)^* (I_\mu\otimes L_X(\Omega))(I\otimes V). 
\end{split}
\end{equation}
 Since $L_X(\Omega)$ is unitarily equivalent to $L_\Omega(X)$ and $X\in \cD_\Omega$, it follows that $L_X(Y)\succeq 0$ and therefore
 $X\in (\mco\{\Omega\})^\circ$ and the proof of the converse part of Theorem \ref{thm:weak-kleski} is complete.

 Suppose  $K^\circ = \cD_\Omega$ for some $d$ and $g$-tuple $\Omega\in \smatdg$ and (without loss of generality)
 $\Omega$ is a minimal defining tuple for $\cD_\Omega$ and write $\Omega=\oplus_{j=1}^N \Omega^j$ where the $\Omega^j$ are irreducible and
 mutually not unitarily equivalent.   By the first part of the proof, $K^\circ =(\mco\{\Omega\})^\circ$ and by the bipolar theorem \cite{EW97} (see also \cite[\S 4.2]{HKMjems}), $K=\mco(\{\Omega\})$. (Here we have  used $0\in K$)
  Evidently $\Omega\in K$ and to  complete the proof it suffices to show $\Omega$ is an Arveson boundary point for $K$.   To prove this statement, suppose $\al\in (\C^d)^g,$  $\be\in \R^g$ and
\begin{equation}
\label{eq:OY}
  Y = \begin{pmatrix} \Om & \al \\ \al^* & \be \end{pmatrix} \in K(d+1)
\end{equation}
  Since $Y\in \mco(\{\Omega\})$, by equation \eqref{eq:mco}
\begin{equation}
 \label{eq:X=V*OV}
  Y = V^* (I_m\otimes \Om) V
\end{equation}
 for some $m$ and isometry $V$.  Equation \eqref{eq:OY} implies $\cD_{Y} \subset \cD_{\Om}$. On the other hand,
 Equation \eqref{eq:X=V*OV} implies $\cD_{Y} \supset \cD_{\Om}$. Hence
$\cD_{Y}=\cD_{\Om}$. Since also $\Omega$ is minimal defining,  the Gleichstellensatz
\cite[Corollary 3.18 or Theorem 3.12]{HKM13} (see also \cite{Zal+}) applies and
there is a unitary $U$ such that
\[
\begin{pmatrix} \Om & \al \\ \al^* & \be \end{pmatrix}
= U^* \bem \Om & 0 \\ 0 & \ga \eem U
\]
for some $\ga$.  In particular, for each $1\le j\le g$,
\[
 \begin{pmatrix} \Om_j & \al_j \\ \al_j^* & \be_j \end{pmatrix}
= U^* \bem \Om_j & 0 \\ 0 & \ga_j \eem U
\]
 By the eigenvalue interlacing result,
 Lemma \ref{lem:interlace},
 it follows   that $\al_j=0$ (and $\ga_j=\be_j$).  Hence, via an application of Lemma \ref{lem:scalar enough for arv}, $\Om$ is
  in the Arveson boundary of $K$.
\hfill\qedsymbol

\section{Free Simplices and Polar Duals of Free Spectrahedra}
\label{sec:FreeSimpAndDuals}

In this section we give a surprising use of
extreme points for
matrix convex sets. Namely,
we show that the polar dual of a free spectrahedron $\cD_A$
can only be
 a free spectrahedron if $\cD_A(1)$ is a polytope (Corollary \ref{cor:hardtobepolar}).
If $\cD_A(1)$ is a simplex, we show
that
$\cD_A^\circ$ is a free spectrahedron (cf.~Theorem \ref{thm:SimplexMain}).

\subsection{The polar dual of a free spectrahedron is seldom a free spectrahedron}
\label{sec:seldom}

In this subsection we use
the theory of
extreme points for
matrix convex sets to
show that if the polar dual of a free spectrahedron $\cD_A$ is a free spectrahedron then $\cD_A(1)$ is a polytope.

\begin{prop}\label{prop:polyhedron}
  Let $K=\cD_A$ be a free  spectrahedron {\CR and assume $\cD_A (2)=\overline{\cD_A (2)}$}. The Euclidean extreme points of $\cD_A(1)$ are
  Arveson boundary points of $\cD_A$.
\end{prop}

\begin{proof}
  Fix a Euclidean extreme point $x\in \cD_A(1)$. In this case $L_A(x) = I-\sum A_jx_j$ has maximal kernel by Corollary \ref{cor:RG+}.
  Suppose $a\in\C^g$, $b\in \R^g$ and
\[
  Y = \begin{pmatrix} x & a \\ a^* & b \end{pmatrix} \in \cD_A(2).
\]
 Using $L(Y)\succeq 0$, it follows that $\ker (\Lambda_A(a)) \supset \ker(L_A(x))$.
 Thus, by Theorem \ref{thm:RG}, $a=0$. By Lemma \ref{lem:scalar enough for arv}, $x$ is an Arveson boundary point of $\cD_A$.
\end{proof}

\begin{remark}
We point out that Proposition \ref{prop:polyhedron}
does not extend to general matrix convex sets.
Indeed, the polar dual $K$ of a free spectrahedron
$\cD_A$ has only finitely many (non-equivalent) Arveson boundary
points in each $K(m)$ by Theorem \ref{thm:weak-kleski}.
However, it is easy to construct examples where
 $K(1)$ has infinitely many
Euclidean extreme points.
 \end{remark}

\begin{cor}
 \label{cor:hardtobepolar}
 If the polar dual of $K=\cD_A$ is again a free spectrahedron{\CR and $\cD_A (2)=\overline{\cD_A(2)}$,} then $\cD_A(1)$ is
 a polyhedron. In particular, if $\cD_A(1)$ is a ball, then the polar of $\cD_A$
 is not a free spectrahedron.
\end{cor}

\begin{proof}
Without loss of generality, $L_A$ is minimal. If $K^\circ=\cD_{B}$ is again a free spectrahedron (with $L_B$ minimal),
then $K$ has finitely many irreducible Arveson boundary points
by Theorem \ref{thm:weak-kleski}.
In particular, by Proposition \ref{prop:polyhedron},
$K(1)$ has finitely many boundary points and is thus a
polyhedron. For the final statement of the corollary,
use that $\cD_A^\circ(1)=\cD_A(1)^\circ$ by
\cite[Proposition 4.3]{HKMjems} since $\cD_A$
is matrix convex.
\end{proof}

\begin{remark}
If $\cC$ is the free cube, then its polar dual
is not a free spectrahedron \cite{HKM13}.
Indeed, its Arveson boundary points
are exactly tuples $J=(J_1,\ldots,J_g)$ of symmetries, i.e., $J_j^*=J_j$ and $J_j^2=I$.
If $\cC^\circ$ were a free spectrahedron, then
$\cC$ would contain only
 finitely many inequivalent irreducible points in its Arveson boundary. But, say for the case $n=2$ (square), each tuple
\[
( \begin{pmatrix} 1 & 0 \\ 0 & -1\end{pmatrix}, \begin{pmatrix} s &t\\ t& -s\end{pmatrix} )
\]
 with $|s|\ne 1$, $t>0$ and $s^2+t^2 =1$ gives such points (at level two).
\end{remark}

We next turn our attention to free simplices, where polar duals are
again free simplices and thus free spectrahedra.

\subsection{Free simplices}

The main result of this subsection is Theorem \ref{thm:SimplexMain}. It characterizes the absolute extreme points of a free simplex $\cD$ (i.e., a matrix convex set $\cD\subset\smatg$ whose scalar points $\cD(1)\subset\R^g$ form a simplex). These are exactly Euclidean extreme points of $\cD(1)$.\looseness=-1

A free spectrahedron $\cD \subset \smatg$ is a \df{free simplex} if it is bounded and  there exists a diagonal tuple $A\in \mathbb S_{g+1}^g$ such that $A$ is a minimal defining tuple for $\cD$. In particular,  $\cD=\cD_A$.

\begin{thm}
\label{thm:SimplexMain}
If $\cD_A$ be a free simplex, then $\free \cD_A=\free(\cD_A)(1)=\euc\cD_A(1)$. Furthermore $\cD_A=\comat (\free \cD_A)$.  Thus, a  point is in the Arveson boundary of $\cD_A$ if and only if it is, up to unitary equivalence,  a direct sum of points  from $\euc\cD_A(1)$. In particular, Arveson boundary points are commuting tuples.
\end{thm}

The proof will use the notation $S^+$ to denote the positive semidefinite elements of a subset $S$ of a matrix algebra.

\begin{proof}
 By boundedness, the set $\{I,A_1,\dots,A_g\}$ is a linearly independent subset of the diagonal matrices in $\mathbb S_{g+1}$. Hence its span is the commutative C$^*$-algebra $\mathscr D$ of diagonal matrices in $M_{g+1}(\C)$. In particular, if $Z\in M_n(\mathscr D)^+ =(M_n\otimes \mathscr D)^+$ then $Z$ is in fact in  $\mathscr D^+\otimes M_n^+$. If $X\in \cD_A(n)$, then $L_A(X)\succeq 0$ and therefore $Z=L_A(X) \in \mathscr D^+\otimes M_n^+$.  It follows that there exists a positive integer $N$, $n\times n$ rank one matrices $P_j\in M_n^+$ and $Q_j\in \mathscr D^+$ for $1\le j \le N$  such that $Z=\sum Q_j\otimes P_j$.  For each $j$ there is a tuple $x_j=(x_{j,0},\dots,x_{j,g})\in \R^{g+1}$ with
\[
 Q_j = x_{j,0} I +\sum_{k=1}^g x_{j,k} A_k\succeq 0.
\]
If $x_{j,0}\le 0$, then $\sum_{k=1}^g x_{j,k} A_k \succeq 0$, and the domain $\cD_A(1)$ is unbounded. Hence $x_0> 0$ and, by replacing $P_j$ by $\frac{1}{x_{j,0}}P_j,$ we may assume $x_{j,0}=1$.
Letting $x^j=(x_{j,1},\ldots,x_{j,g})\in\cD_A(1)$, we thus have
\[
Z=\sum_{j} L_A(x^j)\otimes P_j = I\otimes\sum_j P_j
+\sum_{k=1}^g A_k \otimes \sum_j x_{j,k}P_j.
\]
Since $Z=L_A(X)=I\otimes I + \sum_k A_x\otimes X_k$, the linear
independence of   $\{I,A_1,\dots,A_g\}$
implies
\[
\sum_{j=1}^N P_j=I, \qquad \sum_{j=1}^N x_{jk}P_j=X_k, \; k=1,\ldots,g.
\]
Let $P_j=u_ju_j^*$ for $u_j\in\R^n$, $V={\rm col}(u_1^*,\ldots, u_N^*)$,
and $\Xi_k=x_{1k}\oplus\cdots\oplus x_{Nk}$. Then $V$ is an isometry, and
\[
X_k= V^* \Xi_k V.
\]
Furthermore, the tuple $\Xi=(\Xi_1,\ldots,\Xi_g)\in\cD_A$ since
$\Xi=x^1\oplus\cdots\oplus x^N$ and each $x^j\in\cD_A(1)$.
We conclude $\cD_A$ is the matrix convex hull of $\euc \cD_A(1)$. The rest of the theorem now follows easily.
\end{proof}

\begin{cor}
The matrix convex hull $K$ of $g+1$
affine independent points
in $\R^g$ is a free simplex if and only if $0$ is in the interior of $K(1)$.
\end{cor}

Let $\{e_1,\dots,e_{g+1}\}$ denote the standard orthonormal basis in $\C^{g+1}$.  Let $\Na\in \smatgpg$ denote the $g$-tuple with $\Na_j=-e_j e_j^* + \frac1{g+1}e_{g+1}e_{g+1}^*$ for $1\le j\le g$. Thus $X\in \cD_\Na$ if and only if $X_j\ge -I$ for $1\le j\le g$ and $\sum X_j \le (g+1)I$.  We call $\cD_\Na$ the  \df{Naimark spectrahedron}. Because the affine linear mapping $\cT:\smatg\to \smatgone$ defined by $X_j \mapsto \frac 1g (X_j+I)$ for $1\le j\le g$ and $X_{g+1} \mapsto\frac{g+1}{g} X_j$,  implements an affine linear bijection between $\cD_\Na$ and the set
\[
 \cS=\{Y\in \smatgone: Y_j\succeq 0, \, I = \sum Y_j\}.
\]
The set $\cS$  is not  a spectrahedron (and hence not a free simplex), since it doesn't contain $0$, but it is the translation of a free simplex.

\begin{prop}\label{prop:naimark}
 A spectrahedron $\cD_A\subset \smatg$ is a free simplex if and only if it is affine linearly equivalent to the Naimark spectrahedron.
\end{prop}

\begin{proof}
To prove the only if direction, it suffices to show if $\cD_A$ is a free simplex, then it is affine linearly equivalent to
\beq\label{eq:preNaimark}
\mathfrak S =\{Y\in \smatg: Y_j\succeq 0, I\succeq \sum Y_j\}.
\eeq
Let $\bA$ denote the $(g+1)\times g$ matrix whose $j$-th column is the diagonal of $A_j$.
Since $\cD_A$ is bounded, the tuple $A$ is linearly independent. Thus, without loss of generality, the matrix $\bS=(\bA_{k,j})_{j,k=1}^g$, whose $j$-th column of $S$ are the first $g$ entries of the diagonal of $A_j,$   is invertible.
  Let $\lambda = \bS^{-1}{\bfo}$,  where ${\bf 1}\in \R^g$ is the vector with each entry equal to $1$.  Define $\cT:\smatg \to \smatg$ by
\[
 Z =\cT(X)= \bS (\lambda-X) ={\bf 1} + \bS X.
\]
 Thus $\cT$ is an affine linear map and  $Z_j = I+\sum_{s=1}^g \bA_{j,s} X_s$.  Since $\cD_A$ is bounded and non-empty, the same is true of $\cT(\cD_A)$. Further,  $I+\sum_{s=1}^g \bA_{j,s} X_s \succeq 0$ if and only if $Z_j\succeq 0$. Moreover, since
\[
I\succeq  \sum_{s} \bA_{g+1,s} X_s  = \sum_t \left [\sum_{s} \bA_{g+1,s}(\bS^{-1})_{s,t}\right ](Z_t-I),
\]
 $I+\sum_{s} \bA_{g+1,s} X_s \succeq 0$ if and only if $I+\sum_t \beta_t I \succeq \sum \beta_t Z_t$, where  $\beta_t = \sum_s \bA_{g+1,s}(\bS^{-1})_{s,t}.$  Hence $X\in \cD_A$ if and only if
\[
 Z_j \succeq 0, \,  I+\sum \beta_t \succeq \sum \beta_t Z_t.
\]
Since the set of such $Z$, namely $\cT(\cD_A)$ is bounded,  $\beta_t>0$ for each $t.$ Replacing $Z_t$ by $Y_t= \frac{\beta_t}{1+\sum\beta_t}Z_t$ we have mapped $\cD_A$ to $\mathfrak S$ via a bijective  affine linear transformation.

For the converse, suppose $A$ is a minimal defining tuple for $\cD_A$ and $\cD_A$  is affine linearly equivalent to $\cD_\Na$ via an affine linear map $\cT$. Thus, there exists scalars $(b_{j,k})$ and $c_j$ for $1\le j,k\le g$ such that $X\in \cD_\Na$ if and only if the tuple $\cT(X)$ defined by
\[
 \cT(X)_j = \sum_{k=1}^g b_{j,k} X_k + c_j I
\]
is in $\cD_A$. Thus, $X\in \cD_\Na$ if and only if
\[
 I+ \sum_{j=1}^g A_j\otimes \left[\sum_k b_{j,k} X_K + c_jI\right ]  \succeq 0.
\]
 Rearranging,
\[
 [I+\sum c_j A_j] + \sum_{k} \sum_j (A_j b_{j,k}) \otimes X_k  \succeq 0.
\]
 Choosing $X=0$, it follows that $I+\sum c_j A_j \succ 0$. Let $P$ denote the positive square root of this operator and let $B_k = P^{-1} \sum_j (A_j b_{j,k})P^{-1}$. With this notation, $X\in \cD_\Na$ if and only if
\[
 I+ \sum_{k=1}^n B_k\otimes X_k \succeq 0.
\]
Thus $\cD_\Na=\cD_B$. Since $A$ is a minimal defining tuple for $\cD_A$, the tuple $B$ is minimal defining for $\cD_B$. It follows from the Gleichstellensatz (\cite[Corollary 3.18 or Theorem 3.12]{HKM13}; see also \cite{Zal+}), that $\Na$ and $B$ are unitarily equivalent and in particular $B$ (and hence $A$) has size $g+1$ and  $B$ is a tuple of  commuting self adjoint matrices. Moreover, for some scalars $\alpha_k$,
\[
 I - P^{-2} =  \sum \alpha_k B_k.
\]
Hence $P^{-2}$ commutes with each $B_k$ and since $P$ is positive definite, so does $P$.  Consequently the matrices  $B_k^\prime = \sum_{j} b_{j,k}A_j$ commute and since the matrix $(b_{j,k})$ is invertible, $A$ is a tuple of commuting matrices of size $g+1$ and the proof is complete.
\end{proof}

\begin{cor}\label{cor:pdSimplex}
The polar dual of a free simplex is a free simplex.
\end{cor}

\begin{proof}
Assume $\cD_A$ is a free simplex
and let $\omega=\euc \cD_A(1)\subseteq\R^g$.
Then $\omega$ has $g+1$ elements by
Proposition \ref{prop:naimark}. Build
the corresponding diagonal matrices
 $\Omega_j\in\mathbb S_{g+1}$, $j=1,\ldots,g$.
Then
by Theorem \ref{thm:SimplexMain},
$\cD_A=\mco(\{\Omega\})$. Thus by
Theorem \ref{thm:weak-kleski},
$\cD_A^\circ=\cD_\Omega$ is a free simplex.
\end{proof}

In \cite{FNT+}
Fritz, Netzer and Thom use extreme points to investigate when an abstract operator
system has a finite-dimensional concrete realization.
They show that the maximal operator system above a convex set $C\subseteq\R^g$ is a free spectrahedron if
 and only if C is a polyhedron containing $0$ in its interior \cite[Theorem 3.2]{FNT+}. Similarly,
 among such $C\subseteq\R^g$, the minimal operator system above
 $C$ is a free spectrahedron if and only if $C$ is a simplex
 \cite[Theorem 4.7]{FNT+}.

\section{Applications}\label{sec:exA}

In this section we give further applications of the theory of matrix convex sets.
Subsection  \ref{ssec:cube}
characterizes extreme points of the
free cube.
 In Subsection \ref{ssec:ball} the theory of extreme points is applied to the study of free spectrahedra $\cD_A$ such that $\cD_A(1)$ is the unit disk $\{(x,y)\in\R^2: x^2+y^2 \le 1\}$. In  Subsection \ref{ssec:TV} extreme points are used to analyze the matrix convex hull of the TV screen $\{(X,Y): I-X^2-Y^4 \succeq 0\} \subset \mathbb S^2$.

\subsection{Free cube}\label{ssec:cube}
As noted in Section \ref{sec:seldom},
the Arveson boundary points of the free cube
$\cC=\{X\in\smatg: \|X_i\|\leq1 \text{ for }1\leq i\leq g\}$
are exactly tuples $J=(J_1,\ldots,J_g)$ of symmetries.
In this section we show each Euclidean extreme point
of $\cC$ is an Arveson boundary point.

\begin{prop}\label{prop:cube}
Let $\cC\subset\smatg$ be the free cube. Then $\euc\cC=\arv\cC$.
\end{prop}

\begin{proof}
Letting $\{e_1,\ldots,e_g\}$ denote the standard basis
of $\C^g$, and $A_j=(e_je_j^*)\otimes \begin{pmatrix}1 & 0 \\ 0 & -1\end{pmatrix}$ for $1\leq j\leq g$, we have
$\cC=\cD_A$. We shall use the Arveson-like characterization (Corollary \ref{cor:RG+}) of Euclidean
extreme points of $\cD_A$.

Let $X\in\euc\cD_A(n)$. That is, $X\in\cD_A(n)$
is such that $\ker L_A(X)\subset \ker \La_A(Y)$ for
 $Y\in\smatng$  implies $Y=0$. Assume that
  $X_j^2\neq1$ for some $j$. Without loss of generality, $j=1$. Let $0\neq u\in\C^n$ be any vector orthogonal to
  $\ker (I-X_1^2)$ and form the tuple $Y\in\smatng$ by
  $Y_1=uu^*$ and $Y_k=0$ for $k\geq 2$. Then
  $\ker L_A(X)\subset \ker \La_A(Y)$ and $Y\neq0$, violating $X\in\euc\cD_A$. This contradiction shows
  $X_j^2=I$ for all $j$, i.e., $\euc\cC\subseteq\arv\cC$.
  The converse inclusion is obvious.
\end{proof}

 \subsection{Disks}\label{ssec:ball}
In this subsection we study the extreme points of two different free spectrahedra
whose scalar points describe the unit disk.

\subsubsection{Arveson boundary of the Wild Disk}\label{ssec:arvWDisk}
 Consider the pencil
\[
 L_A(x) = \begin{pmatrix} 1 & x_1 & x_2 \\ x_1 & 1 & 0 \\ x_2 & 0 & 1\end{pmatrix}
\]
 The domain $\cD_A$ is the \df{wild disk}.  Note that $X\in \cD_A$ if and only if $p(X)= I-X_1^2 -X_2^2\succeq 0$.
In anticipation of the forthcoming flood of subscripts we change the
notation $(X_1,X_2)$ to $(X,Y)$ here in Subsection \ref{ssec:arvWDisk}
and in Section \ref{ssec:TV}.

We now give an estimate on the
 size of the kernel of $L_A(X,Y)$ for $(X,Y)$ an Arveson boundary point.
Suppose $(X,Y)$ has size $3n$ and $L_A(X,Y) \succeq 0$.  Let $\cK$ denote the kernel of $p(X,Y)$. A straightforward
 computation shows $k\in \cK$ if and only if
\begin{equation}
 \label{eq:kerLAX}
\begin{pmatrix} k & -X k & -Y k\end{pmatrix}^* \in \ker(L_A(X)).
\end{equation}
 Write
 $\mathbb C^{n} = \cK\oplus \cK^\perp.$  Suppose $\gamma=(\alpha,\beta)\in M_{n,2}(\mathbb C^2) = \mathbb C^{n\times 2}$ and
\[
 \ker(L_A(X))\subset \ker(\La_A(\gamma)^*).
\]
 It is readily verified that this inclusion is equivalent to
\begin{equation}
 \label{eq:wild-rels}
 \begin{split}
  (\alpha^* X+\beta^* Y )k & =0 \\
  \alpha^* k & =0 \\
  \beta^* k & = 0
 \end{split}
\end{equation}
 for all $k\in \cK$.  The last two relations imply that $\alpha,\beta \in \cK^\perp$.  Let $P$ denote the projection onto $\cK^\perp$ and
 $Q=I-P$ the projection onto $\cK$. Further, let
 $X_{12} = Q X P$ and similarly for $Y_{12}$.\looseness=-1

\begin{lemma}
 \label{lem:wild-rels}
    There is no nontrivial solution $\gamma= (\alpha,\beta)$ to the system \eqref{eq:wild-rels} (equivalently $(X,Y)$ is in the Arveson boundary
    of $\cD_A$)
    if and only if both $X_{12}$ and $Y_{12}$ are one-to-one and $\rg(X_{12})\cap \rg(Y_{12}) = (0)$.
\end{lemma}

\begin{proof}
 First suppose $0\ne \alpha \in \cK^\perp$ and $X_{12}\alpha =0$. In this case the pair $\gamma= (0\oplus \alpha, 0)$ is a nontrivial solution
 to the system \eqref{eq:wild-rels}.  Thus, both $X_{12}$ (and by symmetry) $Y_{12}$ are one-to-one.

 Next suppose there exists $u,v\in \cK^\perp$ such that $X_{12}u = Y_{12}v \ne 0$.  In this case $\gamma = (0\oplus u, -0\oplus v)$ is
  a nontrivial solution to the system in \eqref{eq:wild-rels}.
\end{proof}

\begin{lemma}
 \label{lem:arvboundary-weak-beans}
 If $(X,Y)$ has size $3N$ and is in the Arveson boundary, then the kernel of $p(X,Y)$ has dimension at least $2N$.
 Thus the dimension of $\cK$ is at least two-thirds the size of $(X,Y)$.
\end{lemma}

\begin{proof}
 Let $m$ denote the dimension of $\cK$.
  From Lemma \ref{lem:wild-rels},
 it follows that the dimensions of $\rg(X_{12})$ and of $\rg(Y_{12})$ are both  $3N-m$ as subspaces of a space of dimension $m$.
 The intersection of these ranges is nonempty if $6N-2m = 2(3N-m)> m$. Hence, if $X$ is in the Arveson boundary, then $6N\le 3m$.
\end{proof}

\begin{prop}
 \label{prop:lift-one}
   If $(X,Y)$ has size $n$ and the dimension of the kernel $\cK$ of $p(X,Y)$ is $n-1$, then either $(X,Y)$ is in the Arveson
   boundary or   $(X,Y)$ dilates
  to a pair $(\tilde{X},\tilde{Y})$ of size $n+1$ which lies in the vanishing boundary of $\cD_A$; i.e., $p(\tilde{X},\tilde{Y})=0$.
\end{prop}

\begin{proof}
  Suppose $(X,Y)$ is
 not in the Arveson boundary.  In this case the analysis above applies and
 since $\alpha,\beta\in \cK^\perp$, they are multiples of a single unit vector $u$. With a slight change of notation, write
 $\alpha = \alpha u$ and $\beta = \beta u$.  Express $X$ (and similarly $Y$) with respect to the decomposition $\cK\oplus \cK^\perp$ as
\[
 X= \begin{pmatrix} X_{11} & X_{12} \\ X_{12}^* & x_{22} \end{pmatrix}.
\]
 The vector $\begin{pmatrix} X_{12}^* & x_{22}\end{pmatrix}^*$ is $Xu$ and  $X_{12} = QXu$, where, as before, $Q$ is the projection
 onto $K$.  The relation $QX\alpha + Q Y\beta =0$ implies that the vectors $X_{12}$ and $Y_{12}$ are co-linear.  Accordingly,
 write $X_{12} = x_{12}e$ and likewise $Y_{12}=y_{12}e$ for some unit vector $e$.  Let $t$ be a real parameter,  and $a,b$ numbers
 all to be chosen soon.  Consider,
\[
\tilde{X} = \begin{pmatrix} X_{11} & x_{12}e & 0 \\ x_{12}e^* & x_{22} & t\alpha  \\  0 & t\alpha & a \end{pmatrix}
\]
 and similarly for $\tilde{Y}$.  We have,
\[
 I -\tilde{X}^2-\tilde{Y}^2 = \begin{pmatrix} 0 & -(x_{12} (X_{11} + x_{22})+ y_{12}(Y_{11}+y_{22}))e & t(x_{12}\alpha + y_{12}\beta)e \\
            * & 1-[x_{12}^2 +x_{22}^2 +y_{12}^2 + y_{22}^2 +t^2 (\alpha^2+\beta^2)] &  t[(x_{22} + a)\alpha+(y_{22} + b)\beta] \\
            * & * & 1-t^2(\alpha^2+\beta^2) - a^2-b^2 \end{pmatrix}.
\]
 The $(1,2)$ entry is $0$ since $(X,Y)\in\cD_A$.  Likewise the $(1,3)$ entry is $0$ by the relations in Equation \eqref{eq:wild-rels}.
 The parameter $t$ is determined, up to sign, by setting the $(2,2)$ entry equal to zero.

 Let $\Gamma$ denote the vector $\begin{pmatrix} \alpha & \beta \end{pmatrix}^*$ and $\Delta$ the vector $\begin{pmatrix} a & b \end{pmatrix}^*$.
 Thus,
\[
 t^2 \|\Gamma\|^2 = 1- [x_{12}^2 +x_{22}^2 +y_{12}^2 + y_{22}^2] \ge 0.
\]
 We hope to find $\Delta$ so that
\begin{equation}
 \label{eq:wildDelta}
\begin{split}
   t^2 \|\Gamma\|^2 & = 1- \|\Delta\|^2 \\
   \langle \Delta,\Gamma\rangle & = \langle \Sigma,\Gamma\rangle,
\end{split}
\end{equation}
 where $\Sigma = \begin{pmatrix} x_{22} & y_{22} \end{pmatrix}^*$.
 Note that the first equation of \eqref{eq:wildDelta} determines the norm of $\Delta$.
  The claim is that these equations have a solution if $\|\Delta\|\ge\|\Sigma\|$. Now,
\[
 \|\Delta\|^2 = 1- t^2\|\Gamma\|^2 = [x_{12}^2 +x_{22}^2 +y_{12}^2 + y_{22}^2] \ge \|\Sigma\|^2.
 \qedhere
\]
\end{proof}

We  do not know if either the matrix convex hull or the closed matrix convex hull of the Arveson boundary of $\cD_A$ is all of $\cD_A$. Since one can  construct $n\times n$ pairs $(X_1,X_2)$ in the Arveson boundary of $\cD_A$  such that $I-X_1^2-X_2^2\neq0,$ the matrix convex hull of the vanishing boundary of $\cD_A$ is not all of $\cD_A$. We do not know if the same is true of the closed matrix convex hull of the vanishing boundary.

\subsubsection{Arveson boundary of the Spin Disk}

Consider the pencil
\[
L_A(x_1,x_2)=\begin{pmatrix}
1+x_1 & x_2 \\ x_2 & 1-x_1
\end{pmatrix}.
\]
The free spectrahedron $\cD_A$ is the \df{spin disk}. It is
the $g=2$ case of a \df{spin ball} \cite{HKMS+,DDSS+}.

\begin{prop}\label{prop:spinBall}
A pair $(X_1,X_2)\in\mathbb S^2$ is in the Arveson
boundary of $\cD_A$ if and only if $X_1$ and $X_2$ commute
and $p(X_1,X_2)=I-X_1^2-X_2^2=0$. Furthermore,
\beq\label{eq:arvSpins}
\mco ( \arv\cD_A)=\cD_A.
\eeq
As an immediate consequence, the set of absolute extreme points  of $\cD_A$ equals
the Euclidean extreme points  of $\cD_A(1)$.
\end{prop}

\begin{proof}
Recall \cite[Proposition 14.14]{HKMS+}:
a tuple $X\in\mathbb S^2$ is in $\cD_A$
if and only if it dilates
to a commuting
pair $T$ of self-adjoint matrices with joint spectrum in
the unit disk $\mathbb D\subseteq\R^2$.

Suppose $X\in\arv\cD_A$. Then by  \cite[Proposition 14.14]{HKMS+},
$X_1$ and $X_2$ commute. Without loss of generality, they are diagonal, say $X_j={\rm diag}(c_{1j},\ldots,c_{nj})$. Clearly, $c_{i1}^2+c_{i2}^2=1$ for all $i$
as otherwise the tuple $X$ will obviously have a nontrivial dilation. In particular, $p(X_1,X_2)=0$.

Conversely, assume $X_1$ and $X_2 $ commute and
$p(X_1,X_2)=0$. Suppose $\hat X\in\cD_A$ with
\[
\hat X_j=\begin{pmatrix}
X_j & \alpha_j \\ \alpha_j^* & \beta_j
\end{pmatrix}
\]
Since $\cD_A$ is contained in the wild disk (see e.g.~\cite[Example 3.1]{HKM13}),
the tuple $\hat X$ is in the wild disk and thus
$p(\hat X_1,\hat X_2)\succeq0$. Compute
\[
\begin{split}
p(\hat X_1,\hat X_2) & = I - \hat X_1^2-\hat X_2^2
=
\begin{pmatrix}
I-X_1^2-X_2^2-\alpha_1\alpha_1^* -\alpha_2\alpha_2^* & *\\
* & *
\end{pmatrix}
 = \begin{pmatrix}
-\alpha_1\alpha_1^* -\alpha_2\alpha_2^* & *\\
* & *
\end{pmatrix}.
\end{split}
\]
By positive semidefiniteness, $\alpha_1\alpha_1^* +\alpha_2\alpha_2^*=0$ and thus $\alpha_1=\alpha_2=0$, i.e.,
$X\in\arv\cD_A$.

Finally, to prove \eqref{eq:arvSpins},
let $X\in\cD_A$ be arbitrary.
By \cite[Proposition 14.14]{HKMS+}, $X$
dilates
to a commuting
pair $T\in\mathbb S^2$ with joint spectrum in
$\mathbb D$.
By diagonalizing we can thus reduce to
 $X\in\cD_A(1)$. By employing usual convex combination, we can write
$X=\sum_{j=1}^2 \la_j Y^j$, where each $\la_j\geq0,$ $Y^j\in\R^2$ with $\sum\la_j=1$ and $\|Y^j\|=1$.
Thus
\[
X= V^* (Y^1\oplus Y^2) V
\]
for the isometry $V=\begin{pmatrix}\sqrt{\la_1} & \sqrt{\la_2}\end{pmatrix}^*$. By the above,
$Y^1\oplus Y^2\in\arv\cD_A$, concluding the proof.
\end{proof}

\begin{remark}\rm
 One way to see that the wild disk and the spin disk are distinct is to observe that any non-commuting pair $(X,Y)\in \mathbb S^2$  satisfying $I-X^2-Y^2=0$ is in the wild disk, but not necessarily in the spin disk. If it were in the spin disk, it would dilate to a commuting pair $(\tilde{X},\tilde{Y})$ in the spin disk and hence in the wild disk. But $(X,Y)$ is in the Arveson boundary of the wild disk and hence this dilation would simply contain $(X,Y)$ as a direct summand and hence $(\tilde{X},\tilde{Y})$ would then not commute.
\end{remark}

\subsection{TV screen $p=1-x^2-y^4$}\label{ssec:TV}
In this section we consider the extreme points
of the matrix convex hull $\mco(\cD_p)$ of the free semialgebraic set $\cD_p$
associated to $p=1-x^2-y^4$, i.e., the TV screen,
\[
\cD_p=\{(X,Y)\in\mathbb S^2:I-X^2-Y^4\succeq0\}.
\]
An important question is how to compute the matrix convex hull of a general free semialgebraic set.  In the commutative case, a standard and well studied approach involves representing the convex hull of a semialgebraic set as a spectrahedrop; i.e.,  the projection of a spectrahedron living in higher dimensions.  The set $\cD_p(1)=\{(x,y)\in \R^2: 1-x^2-y^4 \ge 0\}$ is of course convex. However already $\cD_p(2)$ is not.  Thus $\cD_p$ is arguably the simplest non-convex free semialgebraic set. By using extreme point analysis on a natural free spectrahedrop approximation of  $\mco(\cD_p)$, we will see that the spectrahedrop paradigm behaves poorly in the free setting.

A natural approximation \cite{Lasse} to
the matrix convex hull $\mco(\cD_p)$ of $\cD_p$
is the projection $\cC$ onto $x,y$-space of the matrix convex set
$$\hTV = \{ (X,Y,W)\in\mathbb S^3 : I- X^2 - W^2 \succeq 0, \ \ W \succeq Y^2 \}.
$$
Clearly, \[
\cC=   \{ (X,Y)\in\mathbb S^2 : \ \exists W\in\mathbb S \ \
I- X^2 - W^2 \succeq 0, \ \ W \succeq Y^2 \}.
\]

\blem
\label{lem:TVhull}
\mbox{}\par
\ben[\rm(1)]
\item
$\cC(1)=\cD_p(1)$;
\item
$\cD_p \subset \cC$;
\item
$\cC$ is matrix convex;
\item
\label{it:hTVrep}
$\cC$ is a free spectrahedrop, i.e., a projection
of a free spectrahedron.
\een
\elem

\bep
These  statements are straightforward to prove. The simplest LMI representation
for $\hTV$ in  \eqref{it:hTVrep} is given by
\[
\La= \begin{pmatrix}
 1 & 0 &   x \\
 0 & 1 & w \\
x & w & 1
 \end{pmatrix}
 \oplus
 \begin{pmatrix}
 1 &    y \\
 y  & w
 \end{pmatrix},
\]
i.e., $\cD_\La = \hTV$. This is  seen by using Schur complements.
It is easy to convert this LMI representation
to a monic one \cite{Lasse}.
Let
\[
L_1(x,y,w)=\begin{pmatrix}
 1 &  \gamma  y \\
 \gamma y  & w+\alpha
 \end{pmatrix}
 , \quad
 L_2(x,y,w)=
 \begin{pmatrix}
 1 & 0 &  \gamma ^2 x \\
 0 & 1 & w \\
  \gamma ^2 x & w & 1-2  \alpha w
 \end{pmatrix}
 \]
 where $\al>0$ and $1+\al^2=\ga^4$, and set $L=L_1\oplus L_2$.
 While strictly speaking $L$ is not monic, it contains $0$ in its interior,
 so can be
easily modified to become monic \cite[Lemma 2.3]{HV07}.
\eep

It is tempting to guess that $\cC$ actually is the matrix convex hull of
$\cD_p$, but alas

\begin{prop}\label{prop:TV}
$\mco(\cD_p)\subsetneq\cC$.
\end{prop}

A proof of this proposition appears  in \cite{Lasse} and  is by brute force
verification that the point in equation \eqref{eq:exlasse} below
is in $\cC$ but not $\mco(\cD_p)$.
Here we shall give some conceptual justification
for why $\mco(\cD_p)\not = \cC$ based on our notions of extreme
points. In the process we illustrate some features of the Arveson boundary.

\subsubsection{The vanishing boundary}
  We define
the \df{vanishing boundary} of $\cD_p$ to be all $X \in \cD_p$
making $p(X)=0$.
Let $\cEv\cD_p$ denote the set of all vanishing boundary points.

\begin{lemma}
$\cEv\cD_p\subseteq\arv\cD_p$.
\end{lemma}

\begin{proof}
Assuming $p(X,Y)=0$, suppose
\[
   Z=\begin{pmatrix} X & \al \\ \al^* & A\end{pmatrix}, \qquad
   W=\begin{pmatrix} Y & \be\\ \be^* &  B\end{pmatrix}
\]
 and $p(Z,W)\succeq 0$.  The $(1,1)$ entry of $p(Z,W)$ is
\[
  1-X^2 -\al\al^* -(Y^2+\be\be^*)^2 - \Gamma \Gamma^*
\]
  for some $\Gamma$.  Using $p(X,Y) = 0$ and letting $Q=\be\be^*$, it follows that
\[
 Q Y^2 + Y^2 Q +Q^2 \preceq 0.
\]
   Now apply the trace and note the trace of $QY^2$ is the same as the trace of $YQY$
  so is nonnegative.  The conclusion is that the trace of $Q^2$ is zero and hence $Q=0$ and $\be=0$.
  Thus $-\al\al^*  - \Gamma \Gamma^* \succeq 0$, so $\al=0$.
\end{proof}

\subsubsection{Some \arvb \ points in $\cC$ }

\def\hA{{ \widehat{A} }}
Here we produce classes of interesting points in
$\arv\cC$.    Consider  the set of points $(X,Y,W)$, denoted $\hA$, in $ \hTV$ such that  $I-X^2-W^2 = 0$ and $W-Y^2$ is a rank one positive semidefinite matrix $J$.
  An association of  $\hA$  to the \arvb \ is given in the next proposition.

 \begin{prop} \label{prop:arvsumTV}
Suppose $(X,Y,W)\in \hA$ and $J=W-Y^2\neq0$ (of course, $J$ is rank one and positive semidefinite).
\ben [\rm(1)]
  \item
    \label{it:notinDp}
   If  $(X,Y)\not\in\cD_p$, then
   \ben[\rm (a)]
   \item \label{it:uniq}
  the lift $(X,Y,W)$  of $(X,Y)$ to $\hTV$ is unique,
  i.e., if also   $(X,Y,S)\in\hTV$ then $S =W$;
 \item  \label{it:notTV}
   $(X,Y)\in\arv\cC$ and $(X,Y,W)\in\arv\hTV$;
 \item \label{CnotCp} $(X,Y)\in \cC\setminus \mco(\cD_p)$.
   \een
\item\label{it:inTV}
  If  $(X,Y)\in\cD_p(n)$,
then $(X,Y,W)$ is a not a Euclidean extreme point of $\hTV(n)$ and $(X,Y)$ is not a Euclidean extreme point  of $\cC$.
\een
\noindent
If $J=0$, then
$W=Y^2$, so
 $(X,Y,Y^2)\in\arv\hTV$, $(X,Y)\in\arv\cC$, and $(X,Y)\in\arv\mco(\cD_p)$.
\end{prop}

\subsubsection{Proof of Proposition \ref{prop:arvsumTV}}

We begin the
with several  lemmas.

\begin{lemma}[Paul Robinson private communication]
 \label{lem:plr}
   Suppose $E$ and $P$ are  positive semidefinite matrices and
   $J\succeq0$ is rank one.
   If
    \[
  (E+P)^2 \preceq (E+J)^2,
 \]
  then $P=J$ or $J E =EJ =\lambda J$ for a scalar $\lambda \geq 0$.
\end{lemma}

An amusing case is when $P=0$. Here $  E \preceq E+J,$
and the lemma says $  E^2 \not \preceq (E+J)^2$
unless $JE= EJ= \lambda J$.

\begin{proof}
  The square root function is operator monotone. Hence,
 $(E+P)^2 \preceq (E+J)^2$ implies $E+P\preceq E+J$ and
 thus $P\preceq J$. Since $J$ is rank one, there is a
 $0\le t\le 1$ so that $P=tJ$.  Thus,
\[
  (E+tJ)^2 \preceq (E+J)^2
\]
  from which it follows that
\[
  0\preceq  (1-t)( EJ+JE + (1+t)J^2).
\]
  Since $0\le t\le 1$, it now follows that if $P$ is not equal to $J$ (so $t<1$) that
\[
   0\preceq EJ+JE +2J^2.
\]
  Using  $J^2=sJ$ for some $s>0$, one finds
\begin{equation}\label{eq:Rg1}
  0 \preceq RR^* - \frac{1}{2s^2}(EJ^2E)
\end{equation}
 where $R$ is the rank one matrix
\begin{equation}\label{eg:Rg2}
  R=(\sqrt{2} + \frac{1}{\sqrt{2}s} E)J.
\end{equation}
If $EJ=0$, then the conclusion holds with $\lambda= 0$.
If $EJ\ne 0$, then (\ref{eq:Rg1}) implies
 the range of $R$ is equal to the range of $EJ$.
 However, combining this range equality with (\ref{eg:Rg2}),
 gives the range of $J$ equals the range of $EJ$.
 Thus $EJ$ and $J$ are both
 rank one with the same range and kernel. It follows that
  $EJ$ is a multiple of $J$. Since $E$ and $J$ are both self-adjoint,
  $E$ and $J$ commute. Finally,  $E$ and $J$ are, by hypothesis,
     positive semidefinite. Hence $EJ=\lambda J$  for some $\lambda>0$.
\end{proof}

\begin{lemma}
 \label{lem:paulsays}
 Suppose $(X,Y,W)\in \hA$ and let $J=W-Y^2
 \neq0$. Thus $J$ is rank one and positive semidefinite.   If $(X,Y,S)\in \hTV$, the either
\begin{enumerate}[\rm (A)]
 \item \label{it:A} $S=W$; or
 \item \label{it:B} there is a $\lambda>0$ such that $Y^2J=\lambda J$.
\end{enumerate}
\end{lemma}

\begin{proof}
 Since $(X,Y,S)\in \hTV$, we have $I-X^2\succeq S^2$ and there is a positive semidefinite matrix $P$ such that $S=Y^2+P$.
 On the other hand, $(X,Y,W)\in \hA$ gives $I-X^2=W^2$ and $W=Y^2+J$. Hence $W^2\succeq S^2$. Equivalently,
\[
 (Y^2+J)^2\succeq (Y^2+P)^2.
\]
 By Lemma \ref{lem:plr}, either $P=J$ or $JY^2=Y^2J =\lambda J$ for some $\lambda\ge 0$.
\end{proof}

  We are now ready to prove
  Proposition \ref{prop:arvsumTV} \eqref{it:notinDp} which we restate as a lemma.

\begin{lemma}\label{lem:arvbdrytv}
If  $(X,Y,W)\in \hA$, $(X,Y,S)\in \cC$, but $(X,Y)\not\in\cD_p,$ then $S=W$ and  $(X,Y)\in \arv \cC\setminus \mco(\cD_p)$ and  $(X,Y,W)\in\arv\hTV$.
\end{lemma}

\begin{proof}
From Lemma \ref{lem:paulsays}, there are two cases to consider.
In case \eqref{it:B}, $W^2 \succeq Y^4 + \sigma J$ for a $\sigma \ge 0$ and hence
  $I-X^2 - Y^4 \succeq I-X^2-W^2 = 0$ so that $(X,Y)\in\cD_p$. Thus, case \eqref{it:A} holds; i.e., $S=W$.  In particular,
$I-X^2-S^2=0$.

Suppose  the tuple
\begin{equation}
 \label{eq:tildes}
  \tX =\begin{pmatrix} X& \al \\ \al^*& *  \end{pmatrix}, \ \
   \tY=\begin{pmatrix} Y & \beta \\ \beta^* & Y_* \end{pmatrix},\ \
 \tS=\begin{pmatrix} S & t \\ t^* & S_* \end{pmatrix}.
\end{equation}
 is in $ \hTV.$
Positivity of  the upper left entry of $\tS-\tY^2$ gives $S  \succeq  Y^2 + \beta \beta^*$ and
 positivity of the upper left entry of $I-\tX^2-\tS^2$ gives
\[
 I - X^2 - S^2  \succeq \al \al^*  + t^* t   \succeq 0.
\]
 From what has already been proved, $S=W$ and  $I-X^2-S^2=0$.  Hence $\al=t=0$.

To show $\beta =0$, observe the inequality $ \tS \succeq  \tY^2$ implies
\[
\begin{pmatrix} S & 0 \\ 0 & S_* \end{pmatrix} - \begin{pmatrix} Y^2+\beta \beta^* & Y\beta + \beta Y_* \\ (Y\beta + \beta Y_*)^* & Y_*^2 + \beta^* \beta \end{pmatrix} \succeq 0.
\]
Hence,
\begin{equation}
 \label{eq:Jetc}
  \begin{pmatrix} J- \beta \beta^*  & -(Y\beta + \beta Y_*) \\ -(Y\beta + \beta Y_*)^* & S_* -(Y_*^2 + \beta^* \beta) \end{pmatrix} \succeq 0.
\end{equation}
Since the left top corner is positive and $J$ is rank one,  $\beta \beta^* = aJ$ for some $0 \leq a \leq 1$. If $a=0$, then $\beta = 0$.
Arguing by contradiction, suppose $a\ne 0$.  The inequality of equation \eqref{eq:Jetc} implies there is a matrix $C$ such that
 $Y\beta + \beta Y_* = JC$. Rearranging and multiplying on the right by $\beta^*$ gives
\[
  Y\beta \beta^*  = \beta\big(-Y_*  + \frac{\beta^*}{a}C^* \big) \beta^*.
\]
Since $\beta$ is rank one ($\beta\beta^* = aJ$ and $J$ is rank one), we conclude,
$Y \beta \beta^* = b \beta \beta^*$ and thus $YJ=bJ$. It now follows that $W^2 =(Y^2+J)^2 = Y^4 + 2ab^2 J+J^2\succeq 0$ and hence
\[
 I-X^2-Y^4 \succeq I-X^2 -W^4 \succeq 0.
\]
 Consequently, $(X,Y)\in \cD_P$, a contradiction.

 Summarizing, if $\tX$, $\tY$ and $\tS$ are as in equation \eqref{eq:tildes} and $(\tX,\tY,\tS)\in \hTV$, then $S=W$ and $\alpha=\beta=t=0$.
 It is immediate that $(X,Y,W)\in \arv \hTV$.   To prove $(X,Y)\in \arv \cC$, suppose $\tX$ and $\tY$  are as in equation \eqref{eq:tildes} and $(\tX,\tY)\in \cC$.  Thus, there is a $\tS$ such that $(\tX,\tY,\tS)\in \hTV$.  Express $\tS$ as in equation \eqref{eq:tildes} too. It follows that $\alpha=\beta=0$. Hence $(X,Y)\in\arv \cC$.  Finally, arguing by contradiction, suppose $(X,Y)\in \mco(\cD_p)$. In this case, $(X,Y)\in\arv\mco(\cD_p)$ since $\mco(\cD_p)\subset\cC$ and $(X,Y)\in\arv \cC$. An application of Lemma \ref{lem:SEinCTV} gives the contradiction $(X,Y)\in\cD_p$.
\end{proof}

\begin{proof}[Proof of the remainder of  Proposition \ref{prop:arvsumTV}]
To prove item \eqref{it:inTV}, suppose $(X,Y)\in \cD_p$ and $(X,Y,W)\in \hA$. In particular, $I-X^2-Y^4\succeq 0$ and both $I-X^2-W^2 =0$ and $W-Y^2=J$ is rank one positive semidefinite.  Combining the first two of these equations gives  $W^2 \succeq Y^4$. Using the third,
$$(Y^2+J)^2 \succeq (Y^2)^2.$$
From Lemma \ref{lem:plr}, $Y^2J=\lambda J$. In  particular $Y^2$ commutes with $J$ (and so $Y$ commutes with $J^{1/2}$).
Verify
\[
  \tX =\begin{pmatrix} X& 0 \\ 0 & X  \end{pmatrix}, \ \
   \tY=\begin{pmatrix} Y & J^{1/2} \\ J^{1/2} & -Y \end{pmatrix},\ \
 \tW=\begin{pmatrix} W & 0 \\ 0 & W \end{pmatrix}
\]
is in $\hTV$.
Since  $J \not= 0$, the tuple $(X,Y,W)$ is not a Euclidean extreme point of $\hTV$ by Theorem \ref{thm:intromainext}\eqref{it:Euclidean-geometric}.
Similarly, to conclude $(X,Y)$ is not a Euclidean extreme point of $\cC$ verify
\[
  \tX =\begin{pmatrix} X& 0 \\ 0 & X  \end{pmatrix}, \ \
   \tY=\begin{pmatrix} Y & J^{1/2} \\ J^{1/2} & -Y \end{pmatrix}
\]
is in $\cC$.

Finally, we prove  the final assertion in the proposition.
Suppose
\[
  \tX =\begin{pmatrix} X& \al \\ \al^*& *  \end{pmatrix}, \ \
   \tY=\begin{pmatrix} Y & \beta \\ \beta^* & Y_* \end{pmatrix},\ \
 \tW=\begin{pmatrix} Y^2 & t \\ t^* & W_* \end{pmatrix}
\]
is in $\hTV$. Then $I-\tX^2-\tS^2 \succeq 0$ gives, in the top left corner, that $I - X^2 -\al \al^* - Y^4 - t t^* = -\al \al^* - t t^*\succeq  0$. Thus $\al = t = 0$. Additionally, $\tW \succeq \tY^2$ gives that, examining the top left corner, $-\beta \beta^* \succeq 0$. Thus $\beta =0$. Therefore $(X,Y,Y^2)\in\arv\hTV$.

Now suppose $\tX$ and $\tY$ are given as in equation \eqref{eq:tildes} and $(\tX,\tY)\in \cC$.
There exists a $\tS$ as in equation \eqref{eq:tildes}  such that $(\tX, \tY, \tS)$ is in $\hTV$. Again, by observing the top left corners of the inequalities $I-\tX^2 -\tS^2 \succeq 0$ and $\tS \succeq \tY^2$, it follows that $Y^4 - S^2 = I-X^2-S^2 \succeq \al \al^* + t t^* \succeq 0$ and $S - Y^2 \succeq \beta \beta^* \succeq 0$.
Since the
square root function is matrix monotone, $Y^2 \succeq S \succeq Y^2$. Hence $S = Y^ 2$.
 Now $\beta=\alpha= 0$ from the first part of this proof and therefore $(X,Y)\in\arv\cC$. The final claim is a consequence of Proposition \ref{lem:SEinCTV}. 
\end{proof}

\subsubsection{Proof of Proposition \ref{prop:TV}}
It suffices to show there is an $(X,Y,W)$ satisfying the hypotheses of Proposition \ref{prop:arvsumTV}\eqref{it:notTV}.
To show there  are in fact many $(X,Y,W)$ that
satisfy the hypotheses of the lemma,  choose a rank one positive $J$ and a $Y$ such that and $JY^2+Y^2J-J^2$ is not positive semidefinite and $W=Y^2+J$   is a contraction; i.e., $I-W^2\succeq 0$.
Choose $X$ such that $X^2=I-W^2$. Hence $(X,Y,W)\in \hA$. On the other hand, since
\[
I-X^2-Y^4 = JY^2+Y^2J-J^2 \not\succeq 0,
\]
 $(X,Y)\not\in\cD_p$.
It is easy to actually construct such $X,Y,W$. The choice
\beq
\label{eq:exlasse}
   Y= \sqrt{\mu}  \begin{pmatrix} 1 & 0 \\ 0 & 0 \end{pmatrix}.
\qquad
  W = \mu \begin{pmatrix} 2 & 1 \\ 1 & 1 \end{pmatrix}
  \qquad
  X^2 = 1 - W^2.
\eeq
 with  $\mu$ chosen so that the norm of $W$ is $1$ appears in  \cite{Lasse}.
\hfill\qedsymbol

\linespread{1.0}

\section{Correction to Proposition \ref{prop:polyhedron} and Corollary \ref{cor:hardtobepolar}, posted June 2019. }
\label{sec:RealFreeDuals}

Before proceeding, as an epilogue we mention that the main theorem here on absolute extreme points, Theorem \ref{thm:intromainext} \eqref{it:wearekleski} and Theorem \ref{thm:wearekleski}, while stated over $\C$ is true with nearly the same proof over the real numbers. Details are in the subsequent paper by Evert and Helton \cite{EH+} on arXiv \url{https://arxiv.org/abs/1806.09053}.

We now address an error in Section \ref{sec:FreeSimpAndDuals} in the original version of this manuscript. Namely, Proposition \ref{prop:polyhedron} is incorrect without an additional hypothesis. In fact there exist free spectrahedra with Euclidean extreme points at level one which are not absolute extreme points. As a consequence, Corollary \ref{cor:hardtobepolar} may not be correct without this additional hypothesis. Note that Corollary \ref{cor:hardtobepolar} is also stated in the intro as Corollary \ref{cor:hardtobepolar-intro}. Indeed, it is unknown to the authors if there exists a free spectrahedron whose first level is not a polyhedron and whose polar dual is again a free spectrahedron. The authors thank Tom-Lukas Kriel for alerting us to the error and providing a counter example, see \cite[Example 7.12]{K+}, which is slightly adapted to our context in Example \ref{exa:KrielBadSpec}. 

The error in Proposition \ref{prop:polyhedron} may be fixed by restricting to free spectrahedra which are equal to their complex conjugate at level $2$. This property may be used to insure the Euclidean extreme points at level one of such a free spectrahedron are absolute extreme points, thereby obtaining a correction of Proposition \ref{prop:polyhedron}. As a consequence, if $A$ is a tuple of real symmetric matrices and $\cD_A$ is a free spectrahedron whose polar dual is again a free spectrahedron, then $\cD_A (1)$ must be a polyhedron.

Before continuing we reemphasize that while we work over the complexes, by definition level one of a free spectrahedron is a subset of the set of $g$-tuples of real numbers. This is due to the fact that the elements of a free spectrahedron are $g$-tuples of self-adjoint matrices

\begin{prop}\label{prop:Realpolyhedron}
Let $\cD_A$ be a free  spectrahedron and assume $\cD_A (2)= \overline{\cD_A (2)}$. Then Euclidean extreme points of $\cD_A(1) \subset \R^g$ are
  Arveson boundary points of $\cD_A$.
\end{prop}

\begin{proof}
  Fix a Euclidean extreme point $x \in \cD_A(1) \subset \R^g$. If $x$ is not Arveson extreme then there are tuples $a\in\C^g$ and $b\in \R^g$ such that
\[
  Y = \begin{pmatrix} x & a \\ a^* & b \end{pmatrix}= \begin{pmatrix} x & a \\ \overline{a} & b \end{pmatrix}\in \cD_A(2).
\]
By assumption $\cD_A (2)= \overline{\cD_A (2)}$ so it follows that
\[
 \overline{Y} = \begin{pmatrix} x & \overline{a} \\ a & b \end{pmatrix} \in \cD_A (2). 
\]
The spectrahedron $\cD_A (1)$ is convex, so we find
\[
 \begin{pmatrix} x & \mathrm{Re} (a) \\ \mathrm{Re} (a) & b \end{pmatrix} = 
\frac{1}{2}\left( \begin{pmatrix} x & a \\ \overline{a}  & b \end{pmatrix}+
 \begin{pmatrix} x & \overline{a} \\ a & b \end{pmatrix} \right) 
 \in \cD_A (2). 
\]
Using Theorem \ref{thm:intromainext} \eqref{it:Euclidean-geometric} then shows that $\mathrm{Re} (a)=0$. 

It remains to show that $\mathrm{Im}(a)=0$. A matrix convex set is closed under unitary conjugation, so $Y \in \cD_A$ implies 
\[
\begin{pmatrix}
x & i a \\
- i \overline{a} & b 
\end{pmatrix} =
\begin{pmatrix}
1 & 0 \\
0 & -i
\end{pmatrix}
\begin{pmatrix}
x &  a \\
\overline{a} & b 
\end{pmatrix} 
\begin{pmatrix}
1 & 0 \\
0 & i
\end{pmatrix}\in \cD_A (2)
\]
However, $\mathrm{Re} (a)=0$ implies $i a = -i \overline{a} = -\mathrm{Im} (a) \in \R^g$. We have assumed $x$ is Euclidean extreme so Theorem \ref{thm:intromainext} \eqref{it:Euclidean-geometric} then shows that $\mathrm{Im} (a)=0$. We conclude that $a=0$ and $x$ is an Arveson extreme point of $\cD_A$.
\end{proof}

For the reader's convenience we present a self-contained example. The example we present is a version of the example found in \cite{K+} which has been simplified to better suit our purpose.

\begin{example}
\label{exa:KrielBadSpec}
Let $A=(A_1, \dots, A_4) \in \C^g$ be the tuple defined by
\[
A_1= \bem -2 & 0 & 0 \\
0 & 1 & 0 \\
0 & 0 & 1
\eem \quad \quad \quad \quad
A_2=\bem 1 & 0 & 0 \\
0 & -2 & 0 \\
0 & 0 & 1
\eem 
\]
and
\[
A_3= \bem 0 & -1 & 0 \\
-1 & 0 & 0 \\
0 & 0 & 0
\eem \quad \quad \quad \quad
A_4=\bem 0 & -i & 0 \\
i & 0 & 0 \\
0 & 0 & 0
\eem .
\]
Then one may easily check that $x=(1,0,0,0)$ is a Euclidean extreme point of $\cD_A$. However the tuple $Y$ with entries
\[
Y_1= \bem 1 & 0 \\
0 & 0
\eem 
\quad \quad
Y_2 = \bem 0 &  0 \\
  0 &    1
   \eem
   \quad \quad
Y_3 = \bem 0  & \frac{3}{2}  \\
 \frac{3}{2}   &   0 
   \eem
   \quad \quad
   Y_4 = \bem 0  & \frac{-3i}{2}  \\
 \frac{3i}{2}   &   0 
   \eem   
\]
is a nontrivial dilation of $x$ which is again an element of $\cD_A$. Therefore $x$ is not an Arveson extreme point of $\cD_A$. \hfill\qed
\end{example}

The corrected Proposition \ref{prop:polyhedron} leads to a similar variant on Corollary \ref{cor:hardtobepolar}.

\begin{cor}
 \label{cor:Realhardtobepolar}
Let $A$ be a tuple of real symmetric matrices. If the polar dual of $K=\cD_A$ is again a free spectrahedron, then $\cD_A(1) \subset \R^g$ is
 a polyhedron. In particular, if $\cD_A(1) \subset \R^g$ is a ball, then the polar of $\cD_A$
 is not a free spectrahedron.
\end{cor}

\begin{proof}
Without loss of generality, $L_A$ is minimal. If $K^\circ=\cD_{B}$ is again a free spectrahedron (with $L_B$ minimal),
then $K$ has finitely many irreducible Arveson boundary points
by Theorem \ref{thm:weak-kleski}.
In particular, by Proposition \ref{prop:Realpolyhedron},
$K(1)$ has finitely many Euclidean extreme points and is thus a
polyhedron. For the final statement of the corollary,
use that $\cD_A^\circ(1)=\cD_A(1)^\circ$ by
\cite[Proposition 4.3]{HKMjems} since $\cD_A$
is matrix convex.
\end{proof}


\end{document}